\documentclass[11pt]{amsart}

\usepackage[paper]{preamble}

\newrobustcmd\ind{k}

\newrobustcmd\Ind{I}

\newrobustcmd\link{\lambda}
\newrobustcmd\linkclass{\ell}
\newrobustcmd\Link{L}
\newrobustcmd\LINK{\Lambda}

\begin{document}

\title{Finite index extensions of treeable equivalence relations are treeable}

\author[R.~Chen]{Ruiyuan Chen}
\address[R.~Chen]{
    Faculty of Mathematics, Informatics, and Mechanics,
    University of Warsaw,
    Warsaw, Poland
}
\email{ruiyuan@mimuw.edu.pl}

\author[A.~Tserunyan]{Anush Tserunyan}
\address[A.~Tserunyan]{
    Department of Mathematics and Statistics,
    McGill University,
    Montréal, QC, Canada
}
\email{anush.tserunyan@mcgill.ca}

\author[R.~Tucker-Drob]{Robin Tucker-Drob}
\address[R.~Tucker-Drob]{
    Department of Mathematics,
    University of Florida,
    Gainesville, FL, USA
}
\email{r.tuckerdrob@ufl.edu}

\dedicatory{Dedicated to Alexander Kechris on the occasion of his 80th birthday}



\begin{abstract}
We prove that every countable Borel equivalence relation containing a finite index Borel treeable subequivalence relation is Borel treeable.
This answers a long-standing question of Jackson--Kechris--Louveau.
Our proof is based on a median subalgebra filtration construction on products of trees, together with an extension of the median graph treeing algorithm of Chen--Poulin--Tao--Tserunyan to certain well-behaved increasing unions.
\end{abstract}

\maketitle

\tableofcontents

\section{Introduction}
\label{sec:intro}

In this article we give an affirmative answer to the following question.

\begin{question}\label{question:intro}
Given a finite collection of locally finite trees on the same vertex set, can one canonically combine them into a single tree-like structure, without choosing a preferred tree in the collection?
\end{question}

No geometric compatibility between the trees is assumed: vertices that are adjacent in one tree may be arbitrarily far apart in another.

\cref{question:intro} lies at the heart of the \emph{finite index over treeable} problem for countable Borel equivalence relations.
This problem, first raised by Jackson--Kechris--Louveau \cite[Open~Problem~6.4(B)]{JKL:cber}, asks whether every countable Borel equivalence relation which contains a treeable finite index subequivalence relation is itself treeable.
The problem has remained open even in the index $2$, probability-measure-preserving setting, where a treeing is sought only on an invariant conull set.

Here we resolve this problem positively and in its strongest form by giving a purely Borel construction of ``one tree from many.''

\begin{theorem}[\cref{thm:treeable-finite-index}]
\label{intr:thm:treeable-finite-index}
Let $E \subseteq X^2$ be a countable Borel equivalence relation on a standard Borel space $X$,
and let $F \subseteq E$ be a Borel subequivalence relation.
Suppose $F$ has finite index in $E$ and is Borel treeable.
Then $E$ is Borel treeable.
\end{theorem}

Recall that a \defn{graphing} of a Borel equivalence relation $E$ on a standard Borel space $X$ is a Borel graph on $X$ whose connected components are exactly the $E$-classes. 
A \defn{treeing} of $E$ is a graphing of $E$ that is acyclic. 
We say that $E$ is \defn{Borel treeable} if $E$ admits a treeing.  
A subequivalence relation $F$ of $E$ is said to have \defn{finite index} in $E$ if every $E$-class contains only finitely many $F$-classes; the number of $F$-classes in an $E$-class need not be uniformly bounded across the space.

Among countable Borel equivalence relations, treeable equivalence relations play a role analogous to that played by free groups among groups.
Indeed, a countable Borel equivalence relation is treeable if and only if it is free as a Borel groupoid (see \cite{alvarez:thesis}).

In the measured setting, much of the interest in treeability is motivated by the menagerie of {\bf strongly treeable} groups, i.e., locally compact second countable groups all of whose free probability-measure-preserving actions induce treeable orbit equivalence relations on an invariant conull set. 
This class of groups includes automorphism groups of locally finite trees, $\mathrm{SL}_2(\#R )$, finitely generated elementarily free groups, and all closed subgroups of these groups \cite{CGMTD}. 
\cref{intr:thm:treeable-finite-index} implies the following permanence property for this class of groups, and for two closely related classes (defined in \cref{sec:corollaries}), which is new even in the case of countable discrete groups, and resolves a problem raised by Gaboriau \cite[p.~89]{gaboriau:cost}.

\begin{corollary}[\cref{cor:strongly_treeable}]
\label{intr:cor:strongly_treeable}
Let $G$ be a locally compact second countable group. 
If $G$ contains a finite index open subgroup $H \le G$ that is strongly treeable, then $G$ itself is strongly treeable.
The same holds for measure strong treeability and Borel strong treeability.
In all cases this implies that $G$ has fixed price, provided $G$ is unimodular.
\end{corollary}

\subsection{Large-scale geometry, median graphs and algebras}

Every free Borel action of a free group of finite or countably infinite rank generates a treeable orbit equivalence relation.
Jackson--Kechris--Louveau proved that the same holds for virtually free groups \cite[Proposition~3.4]{JKL:cber} using the celebrated Almost-Stability Theorem of Dicks--Dunwoody \cite{dicks-dunwoody:groups-graphs}.
Since then, the finite index over treeable problem has spurred the development of several geometric methods and ideas in the theory of countable Borel equivalence relations; see e.g.,
\cite[Part~2]{tserunyan:thesis},
\cite{tserunyan:stallings},
\cite{jardónsánchez:treedecomp},
\cite{CPTT:quasitree},
\cite{ishikura:decomp_cohomology}, 
\cite{pichot:almost_equality_classes}.

Of particular relevance to this article are the results of Chen--Poulin--Tao--Tserunyan \cite{CPTT:quasitree}, which show that a countable Borel equivalence relation is treeable whenever it admits a locally finite Borel graphing
whose components are either
\begin{enumerate*}[roman]
\item \label{intr:it:cptt:quasitree}%
quasi-isometric to trees or
\item
of bounded tree-width.
\end{enumerate*}

Note that \cref{intr:it:cptt:quasitree} is an analogue, for equivalence relations, of the fact that a finitely generated torsion-free group that is quasi-isometric to a free group is itself free (see \cite{ghys-delaharpe:hypgrp}).
This does not immediately imply that treeability of equivalence relations is preserved under finite index extensions:
unlike a finitely generated group,
a countable Borel equivalence relation does not possess an intrinsic quasi-isometry type independent of a graphing, 
and a finite index extension need not provide the coarse geometric compatibility one might hope for.
However, the results of \cite{CPTT:quasitree} do already give an affirmative answer to \cref{question:intro} when the path metrics of the trees all lie in the same bi-Lipschitz equivalence class.
In this case, \cite[\S5.B]{CPTT:quasitree} applied to the union of the trees produces a canonical \emph{hyperplane-finite median graph}, from which
\cite[Theorem~3.1]{CPTT:quasitree} provides a single (slightly less canonical, but still Borel definable) tree.
Thus, the difficulty in \cref{question:intro} and the finite index over treeable problem lies in the lack of geometric compatibility between the input trees.

A \defn{median graph} is a higher-dimensional analog of a tree, with a well-behaved convexity structure carved out by \emph{hyperplanes} akin to single-edge cuts in a tree (see \cref{sec:prelim-med} for the precise definitions).
A typical example of a median graph is a product of trees.
Median graphs have been studied extensively in metric geometry and geometric group theory as (1-skeleta of) CAT(0) cube complexes, and they are also well-known in lattice theory as interval-finite \emph{median algebras}.
See \cite{birkhoff-kiss:median}, \cite{isbell:median}, \cite{bandelt-hedlikova:median}, \cite{sageev:cube_complexes}, \cite{roller:median}, \cite{bowditch:median}.

\begin{figure}
\centering
\subcaptionbox{%
    \label{fig:medgph-hyp:square*cube}%
    $(\#Z/2\#Z)^2 * (\#Z/2\#Z)^3$%
}{\includegraphics{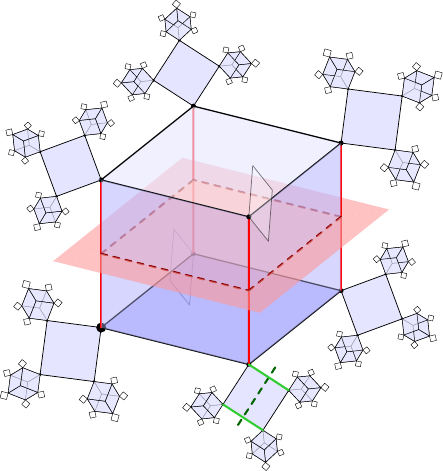}}
\hfil
\subcaptionbox{%
    \label{fig:medgph-hyp:3ble*Z}%
    $(\#Z/2\#Z)^{*3} \times \#Z$%
}{\includegraphics{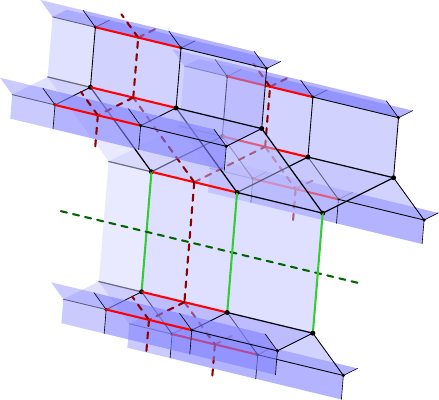}}
\caption{Two median graphs, along with two hyperplanes in each (dashed lines).}
\label{fig:medgph-hyp}
\end{figure}

Among median graphs, the most tree-like are those which are \emph{hyperplane-finite}, i.e., every hyperplane consists of only finitely many edges.
For example, the standard Cayley graph of $(\#Z/2\#Z)^m * (\#Z/2\#Z)^n$ is hyperplane-finite, while the product of an infinite tree with a nontrivial tree is not; see \cref{fig:medgph-hyp}.
Every median graph having a uniform bound on the size of its hyperplanes is quasi-isometric to a tree (by e.g., Manning's bottleneck criterion \cite{manning:quasitree}).
Even without such a uniform bound, the aforementioned \cite[Theorem~3.1]{CPTT:quasitree} produces a spanning tree of a hyperplane-finite median graph (not necessarily quasi-isometric to it) in a Borel manner, which suffices for proving Borel treeability.

\subsection{Proof outline}

Our proof of \cref{intr:thm:treeable-finite-index} is based on a reduction to the machinery of \cite{CPTT:quasitree} using the aforementioned ideas.
The main ingredient is a fully canonical construction of a system of hyperplane-finite median graphs with nested vertex sets, from a finite collection of trees with an arbitrary identification between their vertex sets, without any geometric compatibility hypotheses.
We isolate this core combinatorial result as \cref{thm:tree-link-subalgebra}, which is our positive answer to \cref{question:intro} (as explained in \cref{subsec:functoriality}).

Crucially, this construction uses the alternative representation of median graphs mentioned above: interval-finite \defn{median algebras}, i.e., sets equipped with a ternary ``median'' operation $\ang{x,y,z}$.
(See again \cref{sec:prelim-med}.)
While such algebraic structures are equivalent to median graphs, \emph{a median subalgebra need not be a subgraph; see \cref{fig:medgph-subalg} below}.
The nested system constructed in \cref{thm:tree-link-subalgebra} consists of median graphs that are successively embedded as median subalgebras, and not as subgraphs.
Namely, it comprises a filtration, by disjoint unions of median subalgebras, of the subalgebra generated by the ``diagonal'' in the product of the input trees.

The following is an outline of the main steps in the proof of \cref{intr:thm:treeable-finite-index} (\cref{thm:treeable-finite-index}):
\begin{enumerate}

\item
Using \cite[7.3]{kechris-miller:topics}, one may reduce to the case where there exists a Borel identification (a ``link'') between the $F$-classes contained within each $E$-class
(see proof of \cref{thm:treeable-finite-index}).
Working uniformly over all $E$-classes, this essentially reduces the Borel question to the equivalent form of \cref{question:intro} where the vertex sets of the trees are identified via a link.

\item
Stratify the finitely many input trees into an increasing exhaustive sequence of subforests such that for each pair of input trees, the link restricts to a bi-Lipschitz map between (\textit{typically non-convex}) subsets of each two forest components within these trees (see proof of \cref{thm:tree-link-subalgebra}).

\item
At a given level of the stratification, in each connected component of the product of the finitely many forests, take the median subalgebra generated by the restriction of the link (as a hypergraph); this subalgebra is a hyperplane-finite median graph
(see \cref{sec:medgph-finhyp-link-subalgebra}).
One could perhaps view this as an averaging procedure (in some loose sense) for well-behaved pieces of the input trees.
The disjoint union of these subalgebras, over all connected components of the product, forms one level of the aforementioned filtration.

\item
We then give an extension of the tree-building algorithm from \cite[Theorem~3.1]{CPTT:quasitree} to median algebras equipped with such a filtration by hyperplane-finite median graphs
(see \cref{thm:cptt-union}).

\item 
Applying this to the output of step (3), followed by a standard Voronoi cell argument, yields a single tree on the original vertex set (see \cref{thm:tree-link-treeable}), proving \cref{thm:treeable-finite-index}.

\end{enumerate}

In \cref{sec:corollaries} we deduce some direct consequences of \cref{intr:thm:treeable-finite-index}, including \cref{intr:cor:strongly_treeable}.

\subsection{Single structures and uniform definability}
\label{sec:intro-definability}

All of the results proved before \cref{sec:treeable-finite-index} are formulated for individual structures, rather than for Borel measurable assignments of countable structures to equivalence classes, or for more general Borel fields of such structures.
In this setting, the counterpart of verifying that a construction is Borel is verifying that it is \emph{uniformly definable}, which may be read informally as ``Borel implementable and canonical in the stated inputs.''
This sacrifices no generality and allows us to work directly with the relevant structures (such as products of trees, median subalgebras, or quotient median graphs) without the additional, often cumbersome, notation accompanying the Borel formalism.

The connection between uniform definability and Borel structurability has been used at an informal level for quite some time.
A formal correspondence was established recently by Banerjee--Chen \cite{banerjee-chen:strinterp} (building on the work of Chen--Kechris \cite{chen-kechris:structurable}):
they show that problems in Borel structurability and locally countable Borel combinatorics can be formulated equivalently as uniform definability problems on countable structures equipped with auxiliary combinatorial inputs.
The auxiliary inputs correspond to objects that are always available on equivalence classes in the Borel setting, such as a linear order, or a proper $\#N$-coloring of the intersection graph on finite subsets.
These inputs often serve to break symmetries, allowing choices to be specified by uniform definitions even when no such choices are definable from the original structures alone.

The uniform definability framework also makes fully transparent which auxiliary inputs are used in each construction.
For example, the filtration in \cref{thm:tree-link-subalgebra} is defined from the input forest and link alone and may therefore be regarded as fully canonical,
whereas the tree construction in \cref{thm:cptt} uses the auxiliary input of a linear order on the edge set, together with a proper $\#N$-coloring of the intersection graph on finite sets of edges.

Another benefit of presenting a construction in the framework of uniform definability is that it entails functoriality of the construction under isomorphisms of the inputs.
In the case of \cref{thm:tree-link-subalgebra}, we spell out the resulting functor for trees on a common vertex set in \cref{subsec:functoriality}, giving a precise formulation of our answer to \cref{question:intro}.

At the end of the proofs of \cref{thm:tree-link-subalgebra,thm:cptt,thm:cptt-union}, we indicate why each construction is uniformly definable.
Such a verification is analogous to checking that a construction is Borel when working with the formalism of Borel structurability.
We expect many readers will find these verifications routine, but we include the details to help those who may be unfamiliar with translating constructions into the uniform definability framework.

For background on Borel structurability see \cite{kechris:amenable}, \cite[2.17]{JKL:cber}, \cite{gaboriau:l2}, \cite{seward-tucker-drob:2shift}, \cite{chen-kechris:structurable}, \cite{banerjee-chen:strinterp}.
We have included a short appendix (\cref{sec:Lomega1omega}) on uniform definability, with examples informed by the present article.
The appendix also explains how this framework formalizes constructions, such as quotients, that change the underlying set.

\subsection{Acknowledgments}

It is a pleasure to dedicate this work to our Ph.D.\ advisor, Alekos Kechris.
This is one of many collaborations that would not exist without him.
We continually draw ideas and inspiration from our past discussions with him, not to mention his many books and papers, including of course his work with Jackson and Louveau \cite{JKL:cber} which motivated the present article.
Our BibTeX key for his book \cite{bible} is simply \texttt{bible}.

Alekos's kindness and generosity toward students and colleagues have meant as much to us as his mathematical guidance. 
We are deeply grateful for his support and encouragement over the years.

A.Ts.\ was supported by NSERC Discovery Grant RGPIN-2020-07120.
R.T.D.\ was supported by NSF grant DMS-2246684.

\subsection*{Declaration of AI use}

In researching this article we consulted the large language model ChatGPT-6 Astra Ultra; we now describe the model's mathematical input.
We conceived of the approach to proving \cref{intr:thm:treeable-finite-index} described above prior to consulting GPT. 
This approach builds the final treeing in stages, where the machinery from \cite{CPTT:quasitree} is applied at each stage $n$ yielding a forest.
GPT provided a crucial observation which ensures the coherence of the forest obtained at stage $n+1$ with that obtained at stage $n$.
This observation concerned a ``dual'' formulation of the construction, which we then reconceptualized into the product-of-trees setting described above;
in this current form of the proof, GPT's observation remains implicit but becomes essentially definitional.

In more detail, in the index $2$ setting we had observed that \cite{CPTT:quasitree} could be used to show that the relations $Q_n$ defined in \cref{rmk:Q_nDefinition} are $\link$-invariantly treeable.
We had also observed (using groupoid Bass--Serre theory, as developed by Alvarez \cite{alvarez:thesis}) that each $Q_n$ is freely complemented in $Q_{n+1}$, and that \cref{intr:thm:treeable-finite-index} would follow if each $Q_n$ were moreover $\link$-invariantly freely complemented in $Q_{n+1}$ (see also \cref{rmk:invariant-free-complement}).
Implicit in the application of \cite{CPTT:quasitree} is a natural median graph $M_D$ associated to each $Q_{n+1}$-class $D$, dual to the wallspace on $D$ induced by importing all half-spaces of $T^{\leq n+1}$ via the link $\link$.
GPT supplied the following observation that we missed: for a given $Q_{n+1}$-class $D$, the partial equivalence relation $Q_n|_D$ is invariantly separable by half-spaces in $M_D$.
In our current proof, the corresponding assertion is \cref{thm:tree-link-subalgebra}.(\labelcref{thm:tree-link-subalgebra:invsep}), which follows directly from the definition of the relations $P_n$ in terms of the blocks of the half-space filtration.

The article was entirely written by the three named authors.
After the initial draft was completed, we asked ChatGPT for editorial suggestions and made several minor edits based on the model's feedback.
We take full responsibility for the contents and correctness of the article.

\section{Preliminaries}
\label{sec:prelim}

\subsection{Graphs}
\label{sec:prelim-graph}

By a \defn{graph} $G$ on vertex set $X$, we mean a symmetric, irreflexive binary relation $G \subseteq X^2$.
We will often refer to the graph as $(X,G)$.

Note that $G$ refers to the set of \emph{oriented} edges $(x,y)$.
We write $G/\pm$ for the \defn{unoriented edge set} of $G$.
We will interchangeably regard an unoriented edge as a doubleton set $\set{x,y}$, or as an equivalence class of oriented edges $\set{(x,y), (y,x)}$;
the notation $G/\pm$ is meant to connote the latter representation.
In particular, the latter means that a subset of unoriented edges $F \subseteq G/\pm$ may be equivalently represented as its lift to a subset $\~F \subseteq G$ that is invariant under edge flips.

A graph $(X,G)$ is \defn{locally finite} if the degree of each vertex $x \in X$ is finite.

For a graph $(X,G)$, we write $\#E_G \subseteq X^2$ for the \defn{connectedness relation} of $G$, i.e., the equivalence relation generated by $G$.
We write $X/G := X/\#E_G$ for the \defn{set of connected components} of $G$.

The \defn{path metric} of a connected graph $(X,G)$ is
\begin{eqaligned*}
d_G : X^2 &--> \#N  \\
(x,y) &|--> \min \set{n | \exists \text{ path } x = x_0 \mathrel{G} x_1 \mathrel{G} x_2 \mathrel{G} \dotsb \mathrel{G} x_n = y}.
\end{eqaligned*}

\begin{definition}
\label{def:graph-interval}
For a connected graph $(X,G)$, for $x, y \in X$, the \defn{interval} between them is the union of all geodesic paths from $x$ to $y$, i.e.,
\begin{eqaligned*}
[x,y] := \set{z \in X | d_G(x, y) = d_G(x, z) + d_G(z, y)}.
\end{eqaligned*}
We say $z \in [x,y]$ is \defn{between} $x,y$.

A subset $A \subseteq X$ is \defn{convex} if
$\forall x, y \in A\, ([x, y] \subseteq A)$.
\end{definition}

\subsection{Median graphs and algebras}
\label{sec:prelim-med}

\begin{definition}
\label{def:medgph}
A \defn{median graph} $(X,G)$ is a nonempty graph such that for any $x, y, z \in X$, there is a unique vertex which lies on a geodesic between each pair among $x, y, z$, i.e.,
\begin{eqaligned*}
[x,y] \cap [x,z] \cap [y,z] =: \set{\ang{x,y,z}}
\end{eqaligned*}
is a single vertex $\ang{x,y,z} \in X$, called the \defn{median} of $x, y, z$.
\end{definition}

\begin{example}
Every tree is a median graph.
\end{example}

Intuitively, median graphs are higher-dimensional generalizations of trees.
This is made precise by their equivalent characterization as the 1-skeleta of \emph{CAT($0$) cube complexes}.
Median graphs have also been extensively studied from a variety of other perspectives, such as lattice theory, universal algebra, and Boolean satisfiability.
For detailed surveys, see \cite{bowditch:median}, \cite{roller:median}, \cite{bandelt-hedlikova:median}, \cite{CPTT:quasitree}.
Here, we record the basic facts we will use in this paper.

The ternary median operation on a median graph may be generalized to the following abstract algebraic structure:

\begin{definition}
\label{def:medalg}
A \defn{median algebra} $X = (X, \ang{\blank,\blank,\blank})$ is a set $X$ equipped with a ternary operation
$\ang{\blank,\blank,\blank} : X^3 -> X$,
satisfying the following equational axioms:
\begin{gather*}
\tag{symmetry}
\ang{x,y,z} = \ang{x,z,y} = \ang{y,z,x},  \\
\tag{idempotence}
\ang{x,x,y} = x,  \\
\tag{associativity}
\ang{w,\ang{w,x,y},z} = \ang{w,x,\ang{w,y,z}}.
\end{gather*}
(The last axiom says that the binary operation $\ang{w,\blank,\blank}$ is associative for each fixed $w$.)
\end{definition}

It is readily seen that for every median graph $(X,G)$, the vertex set $X$ equipped with the graph-theoretic median operation becomes a median algebra.
In a general median algebra, intervals and convexity still make sense:

\begin{definition}
\label{def:medalg-interval}
For a median algebra $X$, for $x, y \in X$, the \defn{interval} between them is the image, equivalently the fixed points, of the idempotent map $\ang{x,y,\blank} : X -> X$ (which is idempotent by the associativity and idempotence axioms):
\begin{eqaligned*}
[x,y] := \set{\ang{x,y,z} | z \in X} = \set{z \in X | z = \ang{x,y,z}}. 
\end{eqaligned*}
As in \cref{def:medgph}, it is easily verified that the median $\ang{x,y,z}$ may be recovered as the intersection of the three pairwise intervals.

\defn{Convexity} of $A \subseteq X$ is defined as for graphs (\cref{def:graph-interval}).
\end{definition}

\begin{proposition}[see \cite[16.1.2]{bowditch:median}, \cite[4.3]{bandelt-hedlikova:median}]
\label{thm:medalg-medgph}
Median graphs are equivalent (precisely, $\@L_{\omega_1\omega}$-interdefinable; see \cref{def:Lomega1omega-def}) with nonempty median algebras that are \defn{interval-finite}, i.e., every interval $[x,y]$ is finite.
Given an interval-finite median algebra $X$, the edge relation $G \subseteq X^2$ is recovered via
\begin{eqaligned*}
x \mathrel{G} y  \coloniff  x \ne y \AND [x,y] = \set{x,y}.
\end{eqaligned*}
Henceforth, we identify median graphs with nonempty interval-finite median algebras.
\end{proposition}

\begin{remark}
\label{rmk:medalg-medgph}
Essentially all median algebras of relevance for this article are interval-finite, and hence are median graphs.
However, in our main constructions we make crucial use of certain notions that are inherently algebraic, rather than graph-theoretic, and are most naturally expressed within the category of median algebras.
This includes the following notions:
\end{remark}

\begin{definition}
\label{def:med-subalg}
A \defn{median subalgebra} of a median algebra $X$ is a subset $Y \subseteq X$ that is closed under the operation $\ang{\blank,\blank,\blank}$.

A median subalgebra $Y$ of a median graph $X$ is still interval-finite, hence is a median graph in its own right.
However, \emph{it may not be a subgraph}: edges in $Y$ may not be edges in $X$ (see \cref{fig:medgph-subalg:subalg}).

Every convex subset is a median subalgebra (but not conversely).
On a convex subset of a median graph, the induced median graph is indeed an induced subgraph of the ambient graph.
\end{definition}

\removelastskip
\tikzset{
    every median subalgebra node/.style={
        draw,
        circle,
        DarkGreen,
        very thick,
        inner sep=0pt,
        minimum width=2.5mm,
    },
    every median subalgebra edge/.style={
        draw,
        Green,
        opacity=0.5,
        line width=1.5mm,
        every edge quotes/.append style={DarkGreen, opacity=1},
    },
    hyperplane/.style={
        draw,
        DarkRed,
        very thick,
        dashed,
        every edge quotes/.append style={scriptstyle, inner sep=2pt},
    },
}
\makeatletter
\ExplSyntaxOn
\NewDocumentCommand \medsubalgnode {D(){}} {
    \tl_set:Nn \l__medsubalgnode_name_tl {#1}
    \path[
        every~edge/.style={every~median~subalgebra~edge},
    ] node[
        every~median~subalgebra~node,
        /utils/exec={
            \cs_set_eq:NN \tikz@fig@scan@at \__medsubalgnode_scan_at:w
        },
    ]
}
\cs_new_protected:Npn \__medsubalgnode_scan_at:w at {
    \peek_remove_spaces:n \__medsubalgnode_scan_at_aux:w
}
\cs_new_protected:Npn \__medsubalgnode_scan_at_aux:w ( #1 ) {
    \__medsubalgnode_set_at_aux:w ({#1})
    \tikz@@scan@fig
}
\cs_new_protected:Npn \__medsubalgnode_set_at_aux:w ( #1 ) {
    \tl_if_empty:NF \l__medsubalgnode_name_tl {
        \str_set:Ne \l_tmpa_tl {#1}
        \tl_replace_all:Nnn \l_tmpa_tl {,} {-}
        \tl_set:Ne \tikz@fig@name {\l__medsubalgnode_name_tl \l_tmpa_tl}
    }
    \tikz@scan@one@point \tikz@set@at (#1)
}
\ExplSyntaxOff
\makeatother
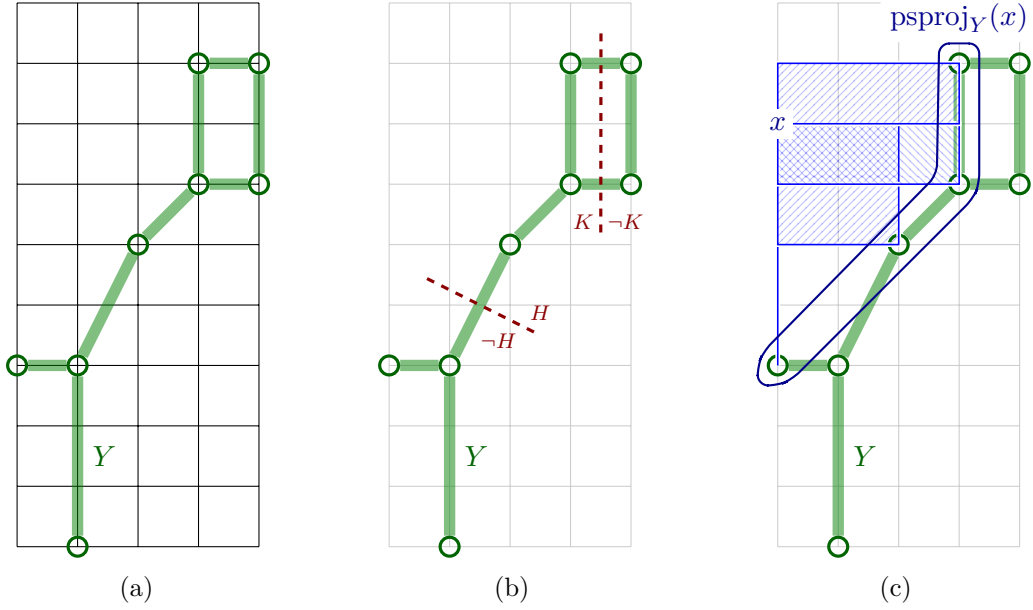
\begin{figure}
\let\VerticalSubalg=\BooleanTrue
\MathDeactivate
\centering
%
%% Subalgebra picture
\subcaptionbox{\label{fig:medgph-subalg:subalg}}
{\begin{tikzpicture}
% flip picture if \VerticalSubalg is true
\IfBooleanT\VerticalSubalg{\tikzset{cm={0,0.8,0.8,0,(0,0)}}}
% ambient grid
\draw (1,0) grid[step=1] (10,4);
% subalgebra nodes
\medsubalgnode(n) at (1,1) {};
\medsubalgnode(n) at (4,0) {};
\medsubalgnode(n) at (4,1) {}
    edge["$Y$"{\IfBooleanTF\VerticalSubalg{right}{above}}] (n1-1)
    edge (n4-0);
\medsubalgnode(n) at (6,2) {} edge (n4-1);
\medsubalgnode(n) at (7,3) {} edge (n6-2);
\medsubalgnode(n) at (7,4) {} edge (n7-3);
\medsubalgnode(n) at (9,3) {} edge (n7-3);
\medsubalgnode(n) at (9,4) {} edge (n7-4) edge (n9-3);
\end{tikzpicture}}%
\hfil
%
%% Hyperplanes picture
\subcaptionbox{\label{fig:medgph-subalg:hyp}}
{\begin{tikzpicture}
% flip picture if \VerticalSubalg is true
\IfBooleanT\VerticalSubalg{\tikzset{cm={0,0.8,0.8,0,(0,0)}}}
% ambient grid
\draw[opacity=0.2] (1,0) grid[step=1] (10,4);
% subalgebra nodes
\medsubalgnode(n) at (1,1) {};
\medsubalgnode(n) at (4,0) {};
\medsubalgnode(n) at (4,1) {}
    edge["$Y$"{\IfBooleanTF\VerticalSubalg{right}{above}}] (n1-1)
    edge (n4-0);
\medsubalgnode(n) at (6,2) {} edge (n4-1);
\medsubalgnode(n) at (7,3) {} edge (n6-2);
\medsubalgnode(n) at (7,4) {} edge (n7-3);
\medsubalgnode(n) at (9,3) {} edge (n7-3);
\medsubalgnode(n) at (9,4) {} edge (n7-4) edge (n9-3);
% hyperplanes
\coordinate (H) at ($(n4-1)!0.5!(n6-2)$);
\draw[hyperplane]
    ($(H)!1cm!90:(n6-2)$)
    to["$H$"'{pos=0.1}, "$\neg H$"{pos=0.1}]
    ($(H)!-1cm!90:(n6-2)$);
\coordinate (K) at ($(n7-3)!0.5!(n7-4)$);
\draw[hyperplane]
    ($(K)!0.8cm!90:(n7-4)$)
    to["$K$"{pos=0.05}, "$\neg K$"'{pos=0.05}]
    ($(K)!-2.6cm!90:(n7-4)$);
\end{tikzpicture}}%
\hfil
%
%% Pseudoprojection picture
\subcaptionbox{\label{fig:medgph-subalg:psproj}}
{\begin{tikzpicture}
% flip picture if \VerticalSubalg is true
\IfBooleanT\VerticalSubalg{\tikzset{cm={0,0.8,0.8,0,(0,0)}}}
% ambient grid
\draw[opacity=0.2] (1,0) grid[step=1] (10,4);
% subalgebra nodes
\medsubalgnode(n) at (1,1) {};
\medsubalgnode(n) at (4,0) {};
\medsubalgnode(n) at (4,1) {}
    edge["$Y$"{\IfBooleanTF\VerticalSubalg{right}{above}}] (n1-1)
    edge (n4-0);
\medsubalgnode(n) at (6,2) {} edge (n4-1);
\medsubalgnode(n) at (7,3) {} edge (n6-2);
\medsubalgnode(n) at (7,4) {} edge (n7-3);
\medsubalgnode(n) at (9,3) {} edge (n7-3);
\medsubalgnode(n) at (9,4) {} edge (n7-4) edge (n9-3);
% position of vertex x that is being pseudoprojected
\coordinate (x) at (8,0);
% draw intervals between x and points in pseudoprojection
\ExplSyntaxOn
\group_begin:
\clist_map_inline:nn {n4-0, n6-2, n7-3, n9-3} {
    \draw[white, double=Blue, line~width=0.6pt, double~distance=0.6pt]
        (x.center) rectangle (#1.center);
    \path[Blue, opacity=0.5, pattern=north~ \bool_if:NTF \l_tmpa_bool {east} {west} ~lines, pattern~color=Blue]
        (x.center) rectangle (#1.center);
    \bool_set_inverse:N \l_tmpa_bool
}
\group_end:
\ExplSyntaxOff
% draw x
\node[fill=white, text=DarkBlue, inner sep=0pt, minimum width=1em, minimum height=1em] at (x) {$x$};
% draw loop around pseudoprojection
\def\r{2ex}  % spacing inside loop
\draw[thick, DarkBlue, rounded corners]
    ($(n9-3) + (0:\r)$)
        node[\IfBooleanTF\VerticalSubalg{above}{right}, DarkBlue, fill=white]{$\psproj_Y(x)$}
    -- ($(n9-3) + (45:{sqrt(2)*\r})$)
    -- ($(n7-3) + (90:\r)$)
    -- ($(n7-3) + (135:\r)$)
    -- ($(n4-0) + (135:\r)$)
    -- ($(n4-0) + (180:\r)$)
    |- ($(n4-0) + (-90:\r)$)
    -- ($(n4-0) + (-45:\r)$)
    -- ($(n7-3) + (-67.5:{(1+sqrt(2))/2*\r})$)
    -- ($(n9-3) + (-45:{sqrt(2)*\r})$)
    -- cycle;
\end{tikzpicture}}%

\caption{%
    \subref{fig:medgph-subalg:subalg}
    A median subalgebra $Y$ (circled vertices) of a grid median graph,
    and the edges of the corresponding median graph on that subalgebra (thick lines);
    together with
    \subref{fig:medgph-subalg:hyp}
    two hyperplanes (dashed lines) within $Y$;
    and
    \subref{fig:medgph-subalg:psproj}
    the pseudoprojection (\cref{thm:medgph-psproj}) of a vertex $x$ onto $Y$, and all minimal intervals between them.
}
\label{fig:medgph-subalg}
\end{figure}

\begin{definition}
\label{def:med-prod}
The \defn{product} of two median algebras $X, Y$ is the Cartesian product $X \times Y$ with the coordinatewise median operation.
Similarly for an arbitrary product $\prod_{i \in I} X_i$ of an indexed family of median algebras $(X_i)_{i \in I}$.

For a finite family of graphs $(X_i,G_i)_{i \in I}$, their \defn{box product graph} $\bigboxprod_{i \in I} G_i$ on the vertex set $\prod_{i \in I} X_i$ is given as follows:
for two $I$-ary tuples $\→x = (x_i)_{i \in I}$ and $\→y = (y_i)_{i \in I}$ in $\prod_{i \in I} X_i$,
\begin{eqaligned*}
\→x \mathrel{(\textstyle\bigboxprod_{i \in I} G_i)} \→y
\coloniff  \exists i \in I\, (x_i \mathrel{G_i} y_i \AND \forall j \ne i\, (x_j = y_j)).
\end{eqaligned*}
When each $G_i$ is a median graph, then their box product is also median, and corresponds to the product median algebra $\prod_{i \in I} X_i$.
The box product is the only ``product'' of graphs considered in this paper; we henceforth simply call it the \defn{product graph}.

For example, a $\#Z^2$-grid is the product of two copies of a $\#Z$-line.
\end{definition}

\begin{definition}
\label{def:med-homom}
A \defn{median homomorphism} $f : X -> Y$ between two median algebras is a map that is median-preserving, i.e., satisfies, for all $x_1, x_2, x_3 \in X$,
\begin{eqaligned*}
f(\ang{x_1, x_2, x_3}_X) = \ang{f(x_1), f(x_2), f(x_3)}_Y.
\end{eqaligned*}
Equivalently by \cref{def:medalg-interval}, $f$ preserves betweenness, i.e.,
\begin{eqaligned*}
x_3 \in [x_1, x_2]_X  \implies  f(x_3) \in [f(x_1), f(x_2)]_Y.
\end{eqaligned*}
Examples include the inclusion of a median subalgebra $X `-> Y$,
and the projections $\pi_i : \prod_j X_j -> X_i$ from a product of median algebras.

Note that \emph{a median homomorphism between median graphs need not be a graph homomorphism} (e.g., the inclusion of the subalgebra in \cref{fig:medgph-subalg}).
\end{definition}

Convex sets in median graphs are well-behaved, resembling the behavior in trees:

\begin{proposition}[see \cite[11.3.3]{bowditch:median}, \cite[2.14, 2.18]{CPTT:quasitree}]
\label{thm:med-proj}
Let $(X,G)$ be a median graph, $\emptyset \ne C \subseteq X$ be convex.
Then for every $x \in X$, there is a unique element of $C$, called the \defn{projection}
\footnote{The term \emph{gate} is also used in the literature.}
$\proj_C(x) \in C$
of $x$ onto $C$,
that is between $x$ and every element of $C$.
Moreover,
\begin{eqaligned*}
\proj_C : X &-->> C \subseteq X
\end{eqaligned*}
is an idempotent median homomorphism fixing $C$.
\end{proposition}

\begin{example}[see \cite[3.2.2, 7.4.4]{bowditch:median}, \cite[2.15]{CPTT:quasitree}]
\label{ex:med-proj-interval}
In any median graph,
every interval $[x, y]$ is convex,
with $\proj_{[x, y]}(\blank) = \ang{x, y, \blank}$.
\end{example}

\begin{proposition}[see \cite[7.3.5]{bowditch:median}, \cite[2.20]{CPTT:quasitree}]
\label{thm:med-proj-proj}
Let $(X,G)$ be a median graph, $C, D \subseteq X$ be nonempty convex.
Then
\begin{eqaligned*}
C \cap D \ne \emptyset
&\iff  \proj_C(D) \cap D \ne \emptyset
\iff  \proj_C(D) \subseteq D  \\
&\iff  \proj_C \circ \proj_D = \proj_D \circ \proj_C = \proj_{C \cap D}.
\end{eqaligned*}
\end{proposition}

\begin{proposition}
\label{thm:medgph-psproj}
Let $(X,G)$ be a median graph, $\emptyset \ne Y \subseteq X$ be a median subalgebra.
Then for every $x \in X$, there are only finitely many $y \in Y$ such that $[x, y] \cap Y = \set{y}$.
We call the set of these $y$ the \defn{pseudoprojection of $x$ onto $Y$}, denoted $\psproj_Y(x) \subseteq Y$.
\end{proposition}

See \cref{fig:medgph-subalg:psproj}.
Note that if $y \in Y$ minimizes the distance $d_G(x, y) = d_G(x, Y)$,
then clearly $y \in \psproj_Y(x)$.
Thus in particular, $\psproj_Y(x) \ne \emptyset$.

\begin{proof}[qedof=\cref{thm:medgph-psproj}]
Fix any $y \in Y$.
It suffices to show that the map
$\proj_{[x,y]} : \psproj_Y(x) -> [x, y]$
is injective.

\begin{claim*}
Let $a \in \psproj_Y(x)$ and $b \in Y$.
Then
\begin{eqaligned*}
\proj_{[x,y]}(b) \in [x, \proj_{[x,y]}(a)]  \implies  a \in [y, b].
\end{eqaligned*}
\end{claim*}
\begin{proof}[of=Claim]
We have
$\proj_{[x,y]}(a) \in [x,a]$
by definition of $\proj_{[x,y]}$,
and likewise
$\proj_{[x,y]}(b) \in [y,b]$;
thus
$\proj_{[x,y]}(b) \in [x, \proj_{[x,y]}(a)]
\implies  \proj_{[x,y]}(b) \in [x,a] \cap [y,b]
\implies  \ang{y, b, a} = \proj_{[y,b]}(a) \in [x,a]$
by \cref{thm:med-proj-proj},
whence $a = \ang{y, b, a}$ since $\ang{y, b, a} \in Y$ and $a \in \psproj_Y(x)$.
\end{proof}

If now $a, b \in \psproj_Y(x)$ with $\proj_{[x,y]}(a) = \proj_{[x,y]}(b)$,
then applying the claim twice yields $a \in [y,b]$ and $b \in [y,a]$,
whence $a = b$.
\end{proof}

\begin{remark}
The proof above shows that the cardinality of $\psproj_Y(x)$ is bounded by
$\min_{y \in Y} \abs{[x, y]}$,
which is in turn bounded by $2^{d_G(x, Y)}$
(by a half-space separation argument; see \cref{thm:med-cvx-sep} below and \cite[2.32]{CPTT:quasitree}).
\end{remark}

\begin{definition}
\label{def:med-half}
A \defn{half-space} in a median algebra $X$ is a convex, coconvex subset $H \subseteq X$.

Equivalently, $H$ is a half-space iff its indicator function $X -> 2$ is a median homomorphism, where $2 = \set{0,1}$ is the unique two-element median algebra (or the single edge graph).

We write $\Half(X) \subseteq 2^X$ for the set of half-spaces of $X$, including the \defn{trivial} half-spaces $\emptyset, X$.

Note that the complement $\neg H := X \setminus H$ of a half-space $H$ is a half-space.
\end{definition}

\begin{remark}
For a median homomorphism $f : X -> Y$ between median algebras, preimage $f^{-1}$ is easily seen to preserve half-spaces.
(Conversely, if $f^{-1}$ preserves half-spaces then $f$ is a median homomorphism, as can be seen using the following result.)
\end{remark}

\begin{proposition}[see \cite[8.1.2]{bowditch:median}, \cite[2.33]{CPTT:quasitree}]
\label{thm:med-cvx-sep}
Any two disjoint convex sets in a median algebra may be separated by a half-space.

In particular, any two distinct vertices are separated by a half-space.
\end{proposition}

\begin{proposition}[see \cite[11.5.1]{bowditch:median}, \cite[2.23--25]{CPTT:quasitree}]
\label{thm:medgph-half}
Let $(X,G)$ be a median graph.
\begin{enumerate}[alph]
\item
For any $(x,y) \in G$, there is a unique half-space $H$ with $x \in H \not\ni y$, namely
$H := \proj_{\set{x,y}}^{-1}(x)$.
\item
Conversely, any half-space $H$ is equal to $\proj_{\set{x,y}}^{-1}(x)$ for any $(x,y) \in G$ with $x \in H \not\ni y$.
\end{enumerate}
Thus, we have a surjection
\begin{eqaligned*}
G &-->> \Half(X) \setminus \set{\emptyset, x}  \\
(x,y) &|--> \proj_{\set{x,y}}^{-1}(x),
\end{eqaligned*}
which descends to a bijection between nontrivial half-spaces and equivalence classes of oriented edges (namely their edge boundaries).
\end{proposition}

\begin{definition}
\label{def:medgph-hyp}
An \defn{oriented hyperplane} in a median graph is an equivalence class of oriented edges lying on the (oriented) boundary of the same nontrivial half-space, as above.

An \defn{unoriented hyperplane} is the corresponding equivalence class of unoriented edges, or equivalently (see \cref{sec:prelim-graph}), an edge-flip-invariant equivalence class of oriented edges lying on the boundaries of a complementary pair of half-spaces $H, \neg H$.

A median graph is \defn{hyperplane-finite} if each of its (oriented or unoriented) hyperplanes is finite.
\end{definition}

\begin{example}
\label{ex:medgph-hyp-tree}
Oriented hyperplanes in trees $(X,T)$ are single edges $(x,y) \in T$.
The corresponding half-space $H$ with $x \in H \not\ni y$ is the component of $x$ in the subgraph $G \setminus \set{(x,y), (y,x)}$.

Thus, hyperplane-finite median graphs may be thought of as ``thickened trees''
(this was a key idea underlying \cite{CPTT:quasitree}).
\end{example}

\begin{example}
\label{ex:med-half-prod}
In a finite product of median algebras (see \cref{def:med-prod})
$X_1 \times \dotsb \times X_n$,
the half-spaces are precisely all $\pi_i^{-1}(H)$ for $H \in \Half(X_i)$, where $\pi_i$ is the $i$th projection map.

For median graphs, this is easily seen from the definition of the box product graph (\cref{def:med-prod}), in which each edge corresponds to an edge from precisely one factor graph.
\end{example}

\begin{example}
\label{ex:med-half-subalg}
In a median subalgebra $Y \subseteq X$, half-spaces are precisely the restrictions $Y \cap H$ of half-spaces $H$ in $X$.
(This follows from the convex separation principle from \cref{thm:med-cvx-sep}.)

Note that several half-spaces in $X$ may restrict to the same half-space in $Y$; see \cref{fig:medgph-subalg:hyp}.
\end{example}

\begin{definition}
\label{def:med-half-nested}
Two half-spaces $H, K \subseteq X$ in a median algebra $X$ are \defn{nested}
if one of $H, \neg H$ is contained in one of $K, \neg K$,
or equivalently,
one of $H, \neg H$ is disjoint from one of $K, \neg K$.

The four intersections
$H \cap K, \neg H \cap K, H \cap \neg K, \neg H \cap \neg K$
are called the \defn{corners} of $H, K$.
\end{definition}

\begin{lemma}
\label{thm:med-half-nested-generator}
If two half-spaces $H, K \subseteq X$ in a median algebra are non-nested, and $X$ is generated (as a median algebra) by a subset $B \subseteq X$, then all four corners of $H, K$ contain an element of $B$.
\end{lemma}
\begin{proof}
The union of any three corners is a median subalgebra of $X$,
because e.g., this is evidently true in a square graph $2 \times 2$, and the indicator functions of $H, K$ yield a homomorphism $X -> 2 \times 2$ under which preimage preserves median subalgebras.
Thus if three corners contained all of $B$, then the fourth corner would be empty, meaning $H, K$ are nested.
\end{proof}

\begin{proposition}[see \cite[2.49]{CPTT:quasitree}]
\label{thm:medgph-hyp-fin-nested}
A median graph is hyperplane-finite iff each of its half-spaces is non-nested with only finitely many other half-spaces.
\end{proposition}

\begin{proposition}[see \cite[9.3.1]{bowditch:median}, \cite[2.28]{CPTT:quasitree}]
\label{thm:medgph-half-nested-tree}
A median graph is a tree iff its half-spaces are pairwise nested.
\end{proposition}

Finally, we will also need to consider quotients of median graphs:

\begin{definition}
\label{def:med-subpoc}
Let $X$ be a median algebra.
A \defn{subpocset} of half-spaces $\@H \subseteq \Half(X)$ is a subcollection of half-spaces that contains the trivial half-spaces $\emptyset, X$ and is closed under negation $\neg$.

Given a subpocset $\@H \subseteq \Half(X)$, the \defn{$\@H$-blocks} are the equivalence classes of elements of $X$ inseparable by $\@H$; we call this the \defn{$\@H$-block equivalence relation}:
\begin{eqaligned*}
x \mathrel{\#E_\@H} y  \coloniff  \forall H \in \@H\, (x \in H  \iff  y \in H).
\end{eqaligned*}
We write $X/\@H := X/\#E_\@H$ for the quotient set of $\@H$-blocks.
\end{definition}

\begin{proposition}[see \cite[2.54]{CPTT:quasitree}]
\label{thm:medgph-quot}
Let $(X,G)$ be a median graph, $\@H \subseteq \Half(X,G)$ be a subpocset of half-spaces.
Then the set of $\@H$-blocks $X/\@H$ equipped with the $G$-adjacency relation forms a median graph;
the quotient map $\pi : X ->> X/\@H$ is a median homomorphism;
and preimage $\pi^{-1}$ yields a bijection
\begin{eqaligned*}
\pi^{-1} : \Half(X/\@H) \cong \@H \subseteq \Half(X).
\end{eqaligned*}
\end{proposition}

\begin{definition}
\label{def:med-subpoc-invsep}
Let $X$ be a median algebra, and let $P \subseteq X^2$ be a partial equivalence relation.
A \defn{$P$-invariant half-space of $X$} means a half-space $H \subseteq X$ whose restriction to the domain of $P$ is $P$-invariant;
in other words, each $P$-equivalence class is contained in $H$ or $\neg H$.
We denote the subpocset of these half-spaces by
\begin{eqaligned*}
\Half^P(X) := \set{H \in \Half(X) | \forall (x, y) \in P\, (x \in H \iff y \in H)}.
\end{eqaligned*}
We say $P$ is \defn{invariantly separable by half-spaces in $X$} if every pair of distinct $P$-classes may be separated by a half-space in $\Half^P(X)$,
i.e., distinct $P$-classes belong to distinct $\Half^P(X)$-blocks.
\end{definition}

\subsection{Countable Borel equivalence relations}
\label{sec:prelim-cber}

We refer to \cite{bible} for general descriptive set theory background.

A \defn{countable Borel equivalence relation} $E$ on a standard Borel space $X$ is a Borel equivalence relation $E \subseteq X^2$ with countable equivalence classes.
See \cite{kechris:cber} for background on the theory of countable Borel equivalence relations.

A \defn{Borel homomorphism} from a Borel equivalence relation $E \subseteq X^2$ to a Borel equivalence relation $F \subseteq Y^2$ is a Borel map $f : X -> Y$ that descends to a map between the quotients $X/E -> Y/F$.
A \defn{Borel reduction} is a Borel homomorphism that descends to an \emph{injection} between the quotients.
We also denote a Borel reduction $f$ from $E$ to $F$ by $f : E \le_B F$ or $f : (X,E) \le_B (Y,F)$, and if such $f$ exists, we say $E$ is \defn{Borel reducible} to $F$, denoted $E \le_B F$.
For \emph{countable} Borel equivalence relations $E \subseteq X^2$ and $F \subseteq Y^2$,
a \defn{Borel bireduction} is a Borel reduction $f : (X,E) \le_B (Y,F)$ that descends to a \emph{bijection} between the quotients $X/E \cong Y/F$.
This is equivalent to the image $f(X) \subseteq Y$ being a \defn{complete section} of $F$, i.e., meeting every $F$-class;
or to the existence of a Borel reduction $g : (Y,F) \le_B (X,E)$ such that $f, g$ descend to inverses of each other.
See again \cite{kechris:cber}.

A \defn{selector function} of an equivalence relation $E \subseteq X^2$ is a function assigning to each $C \in X/E$ a unique point in $C$;
a \defn{transversal} of $E$ is the image of a selector, i.e., a subset of $X$ meeting each $E$-class in a unique point.
Recall that every \emph{finite} Borel equivalence relation $F \subseteq X^2$ (i.e., every $F$-class is finite) on a standard Borel space admits a Borel transversal;
such a transversal is then isomorphic to the quotient space $X/F$, which is thus standard Borel.

A \defn{Borel graphing} of a Borel equivalence relation $E \subseteq X^2$ is a Borel graph $G \subseteq X^2$ that generates $E$ as an equivalence relation, i.e., whose connectedness relation $\#E_G$ is $E$.
A \defn{Borel treeing} of $E$ is an acyclic Borel graphing, i.e., one whose components are trees;
$E$ is \defn{Borel treeable} if it admits a Borel treeing.
See \cite[Ch.~10]{kechris:cber}, \cite[\S3]{JKL:cber}, \cite{kechris-miller:topics}.
A subequivalence relation $F \subseteq E$ has \defn{finite index in $E$} if each $E$-class contains only finitely many $F$-classes.

As noted in the introduction, a Borel graphing of a countable Borel equivalence relation $E \subseteq X^2$ is one instance of a general notion of a ``Borel family of countable first-order structures on each $E$-class'', known as a \defn{Borel structuring} of $E$.
Another instance of a Borel structuring relevant to this paper is a Borel family of median algebra structures (\cref{def:medalg}) on each $E$-class.
See \cref{def:Lomega1omega-borel} for the general definition.

There is a precise correspondence, due to \cite{chen-kechris:structurable}, \cite{banerjee-chen:strinterp}, between Borel combinatorial constructions that may be performed on Borel structurings, and ``canonical'' or ``definable'' combinatorial constructions on single countable structures (thought of as living on the $E$-classes).
We adopt this perspective throughout; thus, the bulk of this paper concerns ``definable'' constructions on countable forests, median graphs, etc., rather than Borel structurings by such.

We give a review of all of the needed logical formalism in \cref{sec:Lomega1omega}.
Briefly, we are interested in \defn{uniformly definable} constructions in the \defn{countably infinitary logic $\@L_{\omega_1\omega}$}, meaning constructions that may be defined by a single $\@L_{\omega_1\omega}$ formula $\phi(x_1, \dotsc, x_n)$ across all structures of the relevant type; see \cref{def:Lomega1omega,def:Lomega1omega-def}.
(A simple example is \cref{def:medgph} of the median operation in a median graph; see \labelcref{def:Lomega1omega-def}.)
Such formulas may then be interpreted simultaneously on every class of a countable Borel equivalence relation $E$ equipped with a Borel structuring, to yield the corresponding Borel construction; see \cref{def:Lomega1omega-borel}.
In addition, there are certain non-definable ``auxiliary'' combinatorial structures that are ``freely available'' on every countable Borel equivalence relation, hence may also be used in Borel constructions.
In this article, we elect to work with one choice of basis for such auxiliary structures, namely a \emph{linear ordering together with a proper $\#N$-coloring of the intersection graph on finite subsets}; see \cref{thm:Lomega1omega-borel-LO+color}.

\section{Median subalgebras generated by end-continuous links}
\label{sec:medgph-finhyp-link-subalgebra}

In this section, we give a method for combining finitely many ``tree-like structures'' with matching large-scale geometry into a single ``tree-like structure'', canonically and symmetrically in the input structures.
Here by a ``tree-like structure'', we mean a median graph with finite hyperplanes (\cref{def:medgph-hyp}).
This core combinatorial construction will be extended in the following sections.

\begin{definition}
\label{def:link}
Let $(X_i)_{i \in \Ind}$ be a nonempty family of sets.
A \defn{link} between them will mean a family of bijections
$\link = (\link_{i,j} : X_i \cong X_j)_{i, j \in \Ind}$
that is coherent, i.e., satisfies
\begin{eqaligned*}
\link_{i,i} &= \id_{X_i},  \\
\link_{i,k} &= \link_{j,k} \circ \link_{i,j}.
\end{eqaligned*}
(That is, $\link$ forms a functor, or cocycle, or action, of the indiscrete equivalence relation $\Ind^2$.)

A link $\link$ is equivalently determined by its \defn{orbit equivalence relation} $\Link \subseteq (\bigsqcup_{i \in \Ind} X_i)^2$, where
\begin{eqaligned*}
X_i \ni x \mathrel{L} y \in X_j  \coloniff  \link_{i,j}(x) = y,
\end{eqaligned*}
and also by its \defn{hypergraph} $\LINK \subseteq \prod_{i \in \Ind} X_i$, namely
\begin{eqaligned*}
\LINK := \set[\Big]{
    \→x = (x_i)_{i \in \Ind} \in \prod_{i \in \Ind} X_i
    |
    \forall i, j \in \Ind\, (\link_{i,j}(x_i) = x_j)
}.
\end{eqaligned*}

A \defn{partial link} $(\link_{i,j})_{i,j}$ between $(X_i)_i$ will mean a link between a family of subsets $Y_i \subseteq X_i$.
(Thus each $\link_{i,j}$ is a partial bijection between $X_i, X_j$, with domain $Y_i$ and image $Y_j$;
and the orbit equivalence relation $\Link$ is a partial equivalence relation on $\bigsqcup_i X_i$.)
\end{definition}

\begin{group}
\let\link=f

\begin{definition}
\label{def:endcts}
Let $\link : (X, G_X) -`> (Y, G_Y)$ be a partial function between the vertex sets of two connected locally finite graphs.
We say $\link$ is \defn{end-continuous} if for every boundary-finite $V \subseteq Y$, the sets $\link^{-1}(V), \link^{-1}(\neg V) \subseteq X$ may be separated by a boundary-finite $U \subseteq X$.
\end{definition}

\begin{remark}
The name ``end-continuous'' is chosen because this means equivalently that $\link$ extends continuously to the closure of its domain in the end compactification $\^X$ of $(X, G_X)$.
See e.g., \cite[\S2.G]{CPTT:quasitree}.
However, this connection will not be used in this paper.
\end{remark}
\end{group}

We now have the main result of this section.
For its statement, recall \cref{def:med-prod} of the product of median graphs,
and \cref{def:med-subalg} concerning induced median graphs on median subalgebras of median graphs.

\begin{proposition}
\label{thm:medgph-finhyp-link-subalgebra}
Let $(X_i, G_i)_{i \in \Ind}$ be a finite family of locally finite hyperplane-finite median graphs,
and let $(\link_{i,j} : X_i -`> X_j)_{i, j \in \Ind}$ be a nonempty partial link between their vertex sets,
with associated hypergraph $\LINK \subseteq \prod_{i \in \Ind} X_i$.
Suppose each $\link_{i,j}$ is end-continuous.
Then the median subalgebra
$\ang{\LINK} = \smash{\ang{\LINK}_{\bigboxprod_i G_i}}$
of the product graph
$(\prod_i X_i, \bigboxprod_i G_i)$
generated by $\LINK$ is hyperplane-finite as a median graph.
\end{proposition}

The proof hinges on the case of two graphs, which is encapsulated by:

\begin{group}
\let\link=f

\begin{lemma}
\label{thm:medgph-unifcts-almostnested}
Let $(X, G_X), (Y, G_Y)$ be median graphs,
let $\link : X -`> Y$ be a partial function,
and let $K \subseteq Y$ be a $G_Y$-half-space such that
$\link^{-1}(K), \link^{-1}(\neg K) \subseteq X$
may be separated by a $G_X$-boundary-finite subset $U \subseteq X$.
Then in the median subalgebra
$\ang{\link} = \ang{\link}_{G_X \boxprod G_Y} \subseteq X \times Y$
generated by the graph of $\link$,
the half-space
$\ang{\link} \cap (X \times K)$
is nested with
$\ang{\link} \cap (H \times Y)$
for all but finitely many $G_X$-half-spaces $H \subseteq X$.
\end{lemma}

\begin{proof}[qedof=\cref{thm:medgph-unifcts-almostnested}]
Suppose $\ang{\link} \cap (X \times K)$ is non-nested with $\ang{\link} \cap (H \times Y)$ for $H \in \Half(X)$.
Then by \cref{thm:med-half-nested-generator}, $f$ meets all four corners of $H \times Y$ and $X \times K$:
\begin{align*}
\link \cap (H \times K) &\ne \emptyset,  &
\link \cap (H \times \neg K) &\ne \emptyset,  \\
\link \cap (\neg H \times K) &\ne \emptyset,  &
\link \cap (\neg H \times \neg K) &\ne \emptyset.
\end{align*}
Let $(x,y) \in \link \cap (H \times K)$
and $(x',y') \in \link \cap (H \times \neg K)$.
Then we have
\begin{align*}
x &\in H \cap \link^{-1}(K),  &
x' &\in H \cap \link^{-1}(\neg K).
\end{align*}
Thus, there is a $G_X$-geodesic in $H$ from $x$ to $x'$, which must pass through some vertex $z \in H$ on the boundary of $U$.
Similarly, there is some $z' \in \neg H$ on the boundary of $U$.
But now any $G_X$-geodesic from $z$ to $z'$ contains an edge on the boundary of $H$, which uniquely determines $H$.
Since there are only finitely many possibilities for $z, z'$, and finitely many edges on a geodesic between them, there are only finitely many possibilities for $H$.
\end{proof}
\end{group}

\begin{proof}[of=\cref{thm:medgph-finhyp-link-subalgebra}]
Let
$\pi_i : \prod_k X_k -> X_i$
be the projection onto the $i$th coordinate,
and
$\pi_{i,j} = (\pi_i, \pi_j) : \prod_k X_k -> X_i \times X_j$
be the projection onto the $(i,j)$th coordinates.
Thus $\pi_{i,j}(\LINK)$ is the graph of $\link_{i,j}$.
By \cref{thm:medgph-unifcts-almostnested},
for any $K \in \Half(X_j, G_j)$,
for all but finitely many $H \in \Half(X_i, G_i)$,
we have that
$\ang{\pi_{i,j}(\LINK)} \cap (X_i \times K)$
is nested with
$\ang{\pi_{i,j}(\LINK)} \cap (H \times X_j)$,
whence by taking preimage under the median homomorphism
$\pi_{i,j} : \ang{\LINK} -> \ang{\pi_{i,j}(\LINK)}$,
we get that
$\ang{\LINK} \cap \pi_j^{-1}(K)$
is nested with
$\ang{\LINK} \cap \pi_i^{-1}(H)$.
Since all half-spaces in $\ang{\LINK}$ are of the form $\ang{\LINK} \cap \pi_i^{-1}(H)$ for some $i$ and $H \in \Half(X_i, G_i)$
(\cref{ex:med-half-prod,ex:med-half-subalg}),
it follows that every half-space in $\ang{\LINK}$ is nested with all but finitely many others,
i.e., $\ang{\LINK}$ is hyperplane-finite by \cref{thm:medgph-hyp-fin-nested}.
\end{proof}

We will be applying \cref{thm:medgph-finhyp-link-subalgebra} to links whose end-continuity is derived from the following standard fact in the coarse geometry of graphs;
see e.g., \cite[2.67]{CPTT:quasitree}.

\begin{group}
\let\link=f

\begin{lemma}
\label{thm:gph-coarse-endcts}
Let $\link : (X, G_X) -> (Y, G_Y)$ be a finite-to-one Lipschitz function between connected locally finite graphs.
Then $\link$ is end-continuous.
\end{lemma}

\begin{proof}
Suppose $\link$ is $k$-Lipschitz, where $k \in \#N$.
Let $V \subseteq Y$ have finite boundary,
and let $(x,x') \in G_X$ with $x \in \link^{-1}(V) \not\ni x'$.
Then there is a path of length $\le k$ between $\link(x), \link(x')$, which must pass through the boundary of $V$.
Thus, $\link(x), \link(x')$ are within distance $k$ of the boundary of $V$.
Since $V$ has finite boundary, $G_Y$ is locally finite, and $\link$ is finite-to-one, there are only finitely many possibilities for $x, x'$.
\end{proof}
\end{group}

\section{Filtrations of median algebras generated by tree links}\label{sec:filtration}

We now extend the construction from the preceding section to a method for combining finitely many trees with a link between their vertex sets, with no geometric constraints, into a single ``tree-like structure'' in a weaker sense: a nested system of hyperplane-finite median graphs.
The next section will show how to turn this into a genuine tree.

The uniform definability (i.e., ``canonicity'') of the following construction is its key feature, that will in particular allow it to be carried out in the Borel context (in \cref{sec:treeable-finite-index}); see \cref{sec:prelim-cber,sec:Lomega1omega}.
Recall here also \cref{def:link} of the hypergraph of a link,
\cref{def:med-prod,def:med-subalg} of product median graph and subalgebra,
and \cref{def:med-subpoc-invsep} of invariant separability.

\begin{theorem}
\label{thm:tree-link-subalgebra}
Let $X$ be a nonempty set,
$T \subseteq X^2$ be a locally finite forest with finitely many components,
and $\link = (\link_{A,B} : A \cong B)_{A, B \in X/T}$ be a link between the $T$-components,
with associated hypergraph
$\LINK \subseteq \prod_{A \in X/T} A$.
From this structure $T, \link$ on $X$, we may uniformly define (in $\@L_{\omega_1\omega}$) a
sequence of partial equivalence relations
\begin{eqaligned*}
({=_\LINK}) \subseteq P_1 \subseteq P_2 \subseteq P_3 \subseteq \dotsb \subseteq \paren{\textstyle \prod_{A \in X/T} A}^2,
\end{eqaligned*}
such that $\bigcupup_n P_n$ has a single equivalence class,
and obeying the following additional properties:
\begin{enumerate}[alph]

\item \label{thm:tree-link-subalgebra:subalg}
Each equivalence class $C$ of each $P_n$ is the median subalgebra of the product median graph
$(\prod_{A \in X/T} A, \bigboxprod_{A \in X/T} \eval{T}_A)$
of the components of $T$
generated by
$C \cap \LINK \subseteq \prod_{A \in X/T} A$.

\noindent
(Thus, $\bigcupup_n P_n$ is the indiscrete equivalence relation on the median subalgebra generated by~$\LINK$.)

\item \label{thm:tree-link-subalgebra:finhyp}
Each such $C$ is hyperplane-finite as a median graph.

\item \label{thm:tree-link-subalgebra:invsep}
For each equivalence class $D$ of each $P_{n+1}$, the partial equivalence relation $\eval{P_n}_D \subseteq D^2$ is invariantly separable by half-spaces in $D$.

\end{enumerate}
\end{theorem}

Note that here, $\prod_{A \in X/T} A$ is the Cartesian product of the vertex sets of the components of $T$,
i.e., the set of selector functions $\→x = (x_A)_{A \in X/T}$ picking a vertex $x_A$ in each component $A$ of $T$.
Equivalently, we may identify such a selector $\→x = (x_A)_A$ with its image transversal $\set{x_A}_A$ of $\#E_T$, which is a finite subset of $X$.
Then $\LINK \subseteq \prod_{A \in X/T} A$ consists of the selectors that pick vertices corresponding to each other via the link $\link$,
or the transversals that form an equivalence class of the orbit equivalence relation $\Link$ of $\link$ (\cref{def:link}).
The transversal representation of $\LINK$ yields the cleanest translation into $\@L_{\omega_1\omega}$ formulas uniformly defining this construction;
see \cref{eq:tree-link-subalgebra:def} below.

\begin{proof}
We abbreviate the product median graph of the components of $T$ by
\begin{eqaligned*}
(X_\times, T_\boxprod)
:= \paren[\Big]{
    \prod_{A \in X/T} \nmsp\nmsp A,
    \bigboxprod_{A \in X/T} \nmsp \eval{T}_A
}.
\end{eqaligned*}

For each $1 \le n \in \#N$, define the \emph{$n$th cutoff subforest of $T$} to be
\begin{eqaligned*}
T^{\le n} := \set[\big]{
    (x,y) \in T
    |
    \forall B \in X/T\, \paren[\big]{d_T(\link_{[x]_T,B}(x), \link_{[x]_T,B}(y)) \le n}
}.
\end{eqaligned*}
That is, in each component $A$ of $T$, we keep only the edges $(x,y) \in \eval{T}_A$ such that the corresponding vertices in every other component $B$ according to the link are at distance $\le n$ apart.
Hence, the $\eval{T^{\le n}}_A$-components are the blocks (\cref{def:med-subpoc}) of the subpocset
\begin{eqaligned*}
\@H^{>n}_A := \set[\big]{
    H \in \Half(A, \eval{T}_A)
    |
    \exists (x, y) \in \eval{T}_A \setminus T^{\le n}\, (x \in H \not\ni y)
} \cup \set{\emptyset, A}
\end{eqaligned*}
of half-spaces in $A$ whose boundary edge is \emph{not} in $T^{\le n}$.
Put
\begin{eqaligned*}
T^{\le n}_\boxprod &:= \bigboxprod_{A \in X/T} \eval{T^{\le n}}_A \subseteq T_\boxprod,  \\
\@H^{>n}_\boxprod &:= \bigcup_{A \in X/T} \pi_A^{-1}(\@H^{>n}_A) \subseteq \Half(X_\times, T_\boxprod),
\end{eqaligned*}
where $\pi_A : X_\times -> A$ is the projection onto the $A$th coordinate.
Then the $T^{\le n}_\boxprod$ form an increasing exhaustive sequence of subgraphs of $T_\boxprod$;
and the $\@H^{>n}_\boxprod$ form a decreasing vanishing sequence of subpocsets of $\Half(X_\times, T_\boxprod)$,
whose blocks are the $T^{\le n}_\boxprod$-components.

Let
\begin{eqaligned*}
\ang{\LINK}_{T^{\le n}_\boxprod}
:= \bigsqcup_\substack{B \in X_\times / T^{\le n}_\boxprod \\ B \cap \LINK \ne \emptyset}
    \ang{B \cap \LINK}_{T_\boxprod}
\subseteq X_\times
\end{eqaligned*}
be the union,
over all $T^{\le n}_\boxprod$-components, i.e., $\@H^{>n}_\boxprod$-blocks, $B$ that meet the hypergraph $\LINK$,
of the median subalgebra
$\ang{B \cap \LINK}_{T_\boxprod} \subseteq B$
generated by $B \cap \LINK$.
From \cref{def:med-subalg}, each $\ang{B \cap \LINK}_{T_\boxprod}$ is an interval-finite median algebra, hence carries an induced median graph structure;
while $\smash{\ang{\LINK}_{T^{\le n}_\boxprod}}$ is a disjoint union of median graphs
(it is not a median algebra globally).

Define finally $P_n$ to be the connectedness relation of this disjoint union of median graphs,
i.e., the restriction to ${\ang{\LINK}_{T^{\le n}_\boxprod}}$ of the connectedness relation of $T^{\le n}_\boxprod$, or equivalently of the $\@H^{>n}_\boxprod$-block equivalence relation.
Thus for $\→x, \→y \in X_\times$,
\begin{eqaligned*}
\→x \mathrel{P_n} \→y  \coloniff
    \→x \in \ang{\LINK}_{T^{\le n}_\boxprod}
    \AND \→y \in \ang{\LINK}_{T^{\le n}_\boxprod}
    \AND \→x \mathrel{\#E_{T^{\le n}_\boxprod}} \→y.
\end{eqaligned*}
Since $\bigcupup_n T^{\le n}_\boxprod = T_\boxprod$ is connected,
it follows that $\bigcupup_n P_n$ is the indiscrete relation on $\ang{\LINK}_{T_\boxprod}$.
We check the other desired properties:
\begin{itemize}

\item[\cref{thm:tree-link-subalgebra:subalg}]
is immediate from the definition of $P_n$.

\item[\cref{thm:tree-link-subalgebra:invsep}]
is also straightforward:
if $D$ is a $P_{n+1}$-class,
then $D$ is contained in an $\@H^{>n+1}_\boxprod$-block $B$,
and the $P_n$-classes $C \subseteq D$ are contained in distinct $\@H^{>n}_\boxprod$-blocks $A \subseteq B$,
whence the half-spaces in $\@H^{>n}_\boxprod$ restricted to $D$ separate the $P_n$-classes $C \subseteq D$.

\item[\cref{thm:tree-link-subalgebra:finhyp}]
follows from \cref{thm:medgph-finhyp-link-subalgebra}:
if $C$ is a $P_n$-class,
then $C = \ang{B \cap \LINK}_{T_\boxprod}$ for a $T^{\le n}_\boxprod$-component $B$,
which is a product of $\eval{T^{\le n}}_A$-components
$B_A := \pi_A(B) \subseteq A$
for each $T$-component $A \in X/T$.
Now $B \cap \LINK$ is the hypergraph of the partial link $\link^B$ between $(B_A)_{A \in X/T}$ obtained by restricting $\link$ to these subsets.
We claim that for each $A, A' \in X/T$, the restriction of $\link_{A,A'} : A -> A'$ to a partial function $\link^B_{A,A'} : B_A -`> B_{A'}$, with domain $\pi_A(B \cap \LINK)$ and image $\pi_{A'}(B \cap \LINK)$, is end-continuous (\cref{def:endcts}).
Indeed, firstly, the domain restriction $\eval{\link_{A,A'}}_{B_A} : B_A -> A'$ is $n$-Lipschitz by definition of $B_A$ being a $T^{\le n}$-component,
and injective because $\link$ is a link,
whence it is end-continuous by \cref{thm:gph-coarse-endcts}; further restricting the domain to $\pi_A(B\cap\LINK)$ preserves end-continuity.
Secondly, end-continuity of
$\smash{\eval{\link_{A,A'}}_{\pi_A(B\cap\LINK)}}$ as a partial function from $B_A$ into $A'$
implies end-continuity of its codomain restriction to $B_{A'}$, since $B_{A'} \subseteq A'$ is a convex subtree; this codomain restriction is exactly $\smash{\link^B_{A,A'}}$, whence the claim.
Thus \cref{thm:medgph-finhyp-link-subalgebra} applies to $\link^B$ to show that $\ang{B \cap \LINK}_{T_\boxprod} \subseteq B$ has finite hyperplanes.

\end{itemize}

Finally, we address the uniform definability of $P_n$.
As noted above, we use here the representation of $X_\times$ via transversals of $\#E_T$, in which $\LINK \subseteq X_\times$ consists of the classes of the orbit equivalence relation $\Link$.
A transversal of $\#E_T$ is a finite subset of $X$, which may be represented as an equivalence class of finite ordered tuples.
To say that $P_n \subseteq X_\times^2$ is uniformly definable then means that its lift to a relation on ordered tuples is definable, i.e., for each $k \in \#N$, the predicate
\begin{eqaligned}[t][
\label{eq:tree-link-subalgebra:def}]
\text{``$T$ has $k$ components and }
\set{x_1, \dotsc, x_k} \mathrel{P_n} \set{y_1, \dotsc, y_k}
\text{''}
\end{eqaligned}
may be expressed by an $\@L_{\omega_1\omega}$ formula
in the variables $x_1, \dotsc, x_k, y_1, \dotsc, y_k$
and the symbols $T, \link$.
(This is an instance of a uniformly definable relation on an \emph{imaginary sort};
see \cref{ex:Lomega1omega-finpow}.)

It is now routine to translate the informal description of the above construction into formulas defining \cref{eq:tree-link-subalgebra:def}, in the following order:
\begin{itemize}

\item
We may define
``$\ang{x, y, z}_T = m$'',
i.e., the union of the graphs of the median operations on each component of $T$,
by a $4$-ary formula read off of \cref{def:medgph}
(see \cref{ex:Lomega1omega-graph-metric}).

\item
Using this, we may define the median operation $\ang{\blank,\blank,\blank}_{T_\boxprod}$ of $T_\boxprod$, where we represent elements of $X_\times$ as equivalence classes of tuples, in the same sense as in \cref{eq:tree-link-subalgebra:def},
i.e., by defining the lifted relation on $k$-tuples, when $T$ has $k$ components, by a $4k$-ary formula for each $k$.

\item
We may define $T^{\le n}$ by an evident formula read off of its definition above,
using the provided link $\link$ and the $T$-path metric $d_T$ (see \cref{ex:Lomega1omega-graph-metric}).

\item
Using this, we may define the $T^{\le n}_\boxprod$-connectedness relation $\#E_{T^{\le n}_\boxprod}$ on $X_\times$.

\item
The orbit equivalence relation $\Link$ is easily interdefinable with the link $\link$ (see \cref{ex:Lomega1omega-link}),
using which we may define the set $\LINK \subseteq X_\times$ of $\Link$-classes.

\item
We may now define
``$\→y$ belongs to the median subalgebra
$\ang{[\→x]_{T^{\le n}_\boxprod} \cap \LINK}_{T_\boxprod}$''
(in the same sense as in \cref{eq:tree-link-subalgebra:def}, i.e., we are really defining its lift to $k$-tuples for each $k$),
by taking a countable disjunction over all countably many median algebra terms
(i.e., syntactic expressions built recursively using the operation $\ang{\blank,\blank,\blank}$)
that may be used to generate $\→y$,
and for each such term,
taking an existential quantifier over all possible generators in
$[\→x]_{T^{\le n}_\boxprod} \cap \LINK$.

\item
Finally, $\→x \mathrel{P_n} \→y$ iff they are in the same $T^{\le n}_\boxprod$-component and both belong to
$\smash{\ang{[\→x]_{T^{\le n}_\boxprod} \cap \LINK}_{T_\boxprod}}$;
this is evidently definable from the above.
\qedhere

\end{itemize}
\end{proof}

\begin{remark}\label{rmk:Q_nDefinition}
We point out the following explicit description of the subalgebras constructed above. 
Let $Q_n$ be the equivalence relation on $X$ given by
\[
Q_n := \set[\big]{ (x,y)\in \#E_T | \forall A \in X/T\, \paren[\big]{(\link_{[x]_T,A}(x), \link_{[x]_T,A}(y)) \in \#E_{T^{\leq n}}} } ,
\]
and let $\~{Q}_n$ be the equivalence relation generated by $Q_n$ and $L$, where $L\subseteq X^2$ is the orbit equivalence relation of $\lambda$.
Then each $\~{Q}_n$-class $D$ meets each $T$-component $A$ in a single $Q_n$-class $D\cap A$, and $\lambda$ restricts to a link $\lambda^D$ between $(D\cap A)_{A\in X/T}$.
For each $P_n$-class $C$, its median algebra generating set $C\cap \LINK$ is exactly the hypergraph of such a $\lambda^D$, namely with $D=\bigsqcup_{A\in X/T}\pi_A(C\cap\LINK )$.
\end{remark}

\section{Treeing filtrations of hyperplane-finite median graphs}
\label{sec:cptt-union}

The following is \cite[3.1]{CPTT:quasitree}.
That result was not stated with the precise auxiliary input data and definability conclusion below;
however, an inspection of the proof shows that these are all that is needed.
For the sake of completeness, we will recall the proof.

\begin{theorem}[Chen--Poulin--Tao--Tserunyan]
\label{thm:cptt}
Let $(X,G)$ be a median graph with finite hyperplanes.

Suppose we are also given a linear ordering of the unoriented edge set $G/\pm$,
as well as a proper $\#N$-coloring of the intersection graph on the collection of all finite subsets of $G/\pm$.

From this input structure, we may uniformly define (in $\@L_{\omega_1\omega}$) a spanning subtree $T \subseteq G$.
\end{theorem}

Recall (see \cref{thm:Lomega1omega-borel-LO+color}) that a linear ordering together with a proper $\#N$-coloring of the intersection graph on finite subsets is one possible basis for the non-canonical but Borel combinatorial structures that may always be constructed on the classes of a countable Borel equivalence relation.
Assuming this auxiliary data on $G/\pm$, which is a finite-to-one subquotient of $X^2$, is weaker than assuming it on $X$ (see \cref{thm:Lomega1omega-LO+color-finquot}), as was essentially done in \cite[3.1]{CPTT:quasitree}.

\begin{proof}
Note that the given linear ordering and $\#N$-coloring may be used to $\#N$-color the unoriented hyperplanes, or equivalently, $\neg$-invariantly color the nontrivial half-spaces
\begin{eqaligned*}
\Half(X,G) \setminus \set{\emptyset, X} = \bigsqcup_{n \in \#N} \@H_n,
\end{eqaligned*}
so that any $H, K \in \@H_n$ with the same color $n$ are nested.
Indeed, the unoriented hyperplanes are finite equivalence classes of edges, hence the given linear ordering and $\#N$-coloring may be used to construct the same auxiliary data on the set of unoriented hyperplanes (via lexicographical ordering; see \cref{thm:Lomega1omega-LO+color-finquot} for the definability formulation of this standard fact).
And by hyperplane-finiteness, every half-space is non-nested with only finitely many others (\cref{thm:medgph-hyp-fin-nested});
thus non-nestedness yields a locally finite graph on hyperplanes, which may be countably colored by \cite[4.5]{KST:chromatic} (see \cref{thm:KST-lfgraph-color} for the definability formulation of this result).

So let $\Half(X,G) \setminus \set{\emptyset, X} = \bigsqcup_n \@H_n$ be a coloring as above.
Put $\@H_{\ge n} := \bigcup_{m \ge n} \@H_m$.
We inductively construct an increasing sequence of subforests
$T_0 \subseteq T_1 \subseteq \dotsb \subseteq G$
such that the components of each $T_n$ are the $\@H_{\ge n}$-blocks (\cref{def:med-subpoc}).
The $\@H_{\ge 0}$-blocks are singletons, so start with $T_0 := \emptyset$.
Given $T_n$ whose components are the $\@H_{\ge n}$-blocks, each $\@H_{\ge n+1}$-block $C$ is a disjoint union of $\@H_{\ge n}$-blocks which are separated by the pairwise nested half-spaces in $\@H_n$.
For each two adjacent $\@H_{\ge n}$-blocks $A, B$, all edges between them belong to a unique unoriented hyperplane, which is the boundary of a unique pair of complementary half-spaces;
if these half-spaces are in $\@H_n$, then pick the least edge between $A, B$ on that hyperplane.
Take $T_{n+1}$ to be $T_n$ together with all of these chosen least edges between each two such $A, B$; then $T_{n+1}$ is a forest by the nestedness of $\@H_n$ and \cref{thm:medgph-half-nested-tree,thm:medgph-quot}.
This completes the induction.
Finally, take $T := \bigcupup_n T_n$, which is connected since $\bigcapdown_n \@H_{\ge n} = \emptyset$.

For uniform definability:
we encode a half-space as its boundary hyperplane, which is a finite set of edges
(i.e., formally, an element of the imaginary sort $[G]^{<\omega}$; see \cref{ex:Lomega1omega-finpow}).
The definition of $T_{n+1} \subseteq G$ given $T_n$ above is now straightforward to describe via a formula,
e.g., by defining in turn
\begin{eqgathered*}
\psi(a, b) \iff \text{``$a, b$ belong to adjacent $\@H_{\ge n}$-blocks separated by a half-space in $\@H_n$''},  \\
\theta(a, b, c, d) \iff \text{``$\psi(a, b)$ and $\set{c, d}$ is the least edge between the $\@H_{\ge n}$-blocks of $a, b$''}
\end{eqgathered*}
(where $A, B$ above are the $\@H_{\ge n}$-blocks of $a, b$),
whence
$c \mathrel{T_{n+1}} d \iff c \mathrel{T_n} d \OR \exists a, b\; \theta(a, b, c, d)$.
\end{proof}

\begin{remark}
\label{rmk:cptt-quot}
The auxiliary input data in \cref{thm:cptt} are preserved under passing to the quotient graph of blocks $X/\@H$, for any subpocset of half-spaces $\@H \subseteq \Half(X, G)$ (\cref{thm:medgph-quot}).
Indeed, an edge between two blocks $A, B \in X/\@H$ in the quotient graph means that $A, B$ are $G$-adjacent;
the $G$-edges between them lie on the boundary of a half-space in $\@H$, so there are only finitely many.
Thus the edge set of $X/\@H$ may be identified with a finite-to-one subquotient of $G$,
and so we may construct a linear ordering and proper $\#N$-coloring of finite subsets from that given on $G$,
again via a lexicographical ordering argument (see \cref{thm:Lomega1omega-LO+color-finquot}).

(Note that by contrast, requiring a linear ordering as well as proper $\#N$-coloring of finite subsets of the vertex set $X$, as in \cite[3.1]{CPTT:quasitree}, is \emph{not} preserved under passing to quotients.)
\end{remark}

We now extend \cref{thm:cptt} to certain median graphs, or even median algebras, that do not have finite hyperplanes, at the cost of having to provide coherent witnesses to ``local'' hyperplane-finiteness as part of the input structure.
Recall here \cref{def:med-subpoc-invsep} of invariant separability.

\begin{theorem}
\label{thm:cptt-union}
Let $X$ be a nonempty median algebra equipped with an increasing exhaustive sequence of partial equivalence relations
\begin{eqaligned*}
\emptyset = P_0 \subseteq P_1 \subseteq P_2 \subseteq P_3 \subseteq \dotsb \nearrow X^2,
\end{eqaligned*}
obeying the following properties:
\begin{enumerate}[alph]

\item \label{thm:cptt-union:subalg}
Each equivalence class $C$ of each $P_n$ is an interval-finite median subalgebra of $X$.

\item \label{thm:cptt-union:finhyp}
Each such $C$ is hyperplane-finite as a median graph.

\item \label{thm:cptt-union:invsep}
For each equivalence class $D$ of each $P_{n+1}$, the partial equivalence relation $\eval{P_n}_D \subseteq D^2$ is invariantly separable by half-spaces in $D$.

\end{enumerate}

Suppose we are also given a linear ordering,
as well as a proper $\#N$-coloring of the intersection graph on the collection of all finite subsets,
of the unoriented edge set of the median graph on each equivalence class $C$ of each $P_n$.

From this input structure, we may uniformly define (in $\@L_{\omega_1\omega}$) a tree $T$ on the vertex set $X$,
such that each equivalence class of each $P_n$ is convex in $T$.
\end{theorem}

Note again that the auxiliary input of a linear ordering and proper $\#N$-coloring is implied by having the same structure on the entire set $X$, by the usual lexicographical ordering (\cref{thm:Lomega1omega-LO+color-finquot}).

\begin{proof}
We inductively construct an increasing sequence of forests
$T_0 \subseteq T_1 \subseteq \dotsb \subseteq X^2$
such that each $T_n$ generates $P_n$,
and then take $T := \bigcupup_n T_n$.
Start with $T_0 := \emptyset$.

Suppose a forest $T_n \subseteq P_n$ generating $P_n$ is given.
Let $D$ be an equivalence class of $P_{n+1}$.
Consider the subpocset $\Half^{\eval{P_n}_D}(D) \subseteq \Half(D)$ of $\eval{P_n}_D$-invariant half-spaces in $D$ (\cref{def:med-subpoc-invsep});
by hypothesis \cref{thm:cptt-union:invsep}, the $P_n$-classes contained in $D$ are contained in distinct $\Half^{\eval{P_n}_D}(D)$-blocks.
Let $T_{n+1}$ be $T_n$ extended with the following edges, within each $P_{n+1}$-class $D$:
\begin{itemize}

\item
For each $\Half^{\eval{P_n}_D}(D)$-block $A \subseteq D$ that does not contain a $P_n$-class, apply \cref{thm:cptt} to get a tree on $A$ (using the provided linear ordering and $\#N$-coloring of finite subsets of edges in $D$, restricted to edges in the convex set $A \subseteq D$), and add the edges of this tree to $T_{n+1}$.

\item
For each $\Half^{\eval{P_n}_D}(D)$-block $A \subseteq D$ that contains a (unique) $P_n$-class $C$, by the induction hypothesis, $\eval{T_n}_A$ is a tree on $C$.
Extend this to a tree on $A$ as follows: for each $x \in A \setminus C$, join $x$ via a single edge to $y_x \in C$, where $y_x$ is the endpoint of the lexicographically least geodesic within $A$ (according to the provided ordering of edges in $D$) from $x$ to the pseudoprojection $\psproj_C(x) \subseteq C$ of $x$ onto $C$ (\cref{thm:medgph-psproj}).

\end{itemize}
Thus, we have defined a tree on each $\Half^{\eval{P_n}_D}(D)$-block.
\begin{itemize}

\item
Now apply \cref{thm:cptt,rmk:cptt-quot} to the quotient of $D$ by $\Half^{\eval{P_n}_D}(D)$ to get a spanning subtree of this quotient graph.
For each edge between two blocks $A, B$ in this tree, $A, B$ are adjacent in $D$, with finitely many edges between them lying on a hyperplane by \cref{thm:cptt-union:finhyp};
pick the least such edge (in the provided linear ordering on edges in $D$) and add it to $T_{n+1}$.

\end{itemize}
Thus, $T_{n+1}$ on each $P_{n+1}$-class $D$ is a tree of $\Half^{\eval{P_n}_D}(D)$-blocks, with a single edge chosen between each two adjacent blocks (in the tree), as well as a tree on each block; hence it is a tree on $D$.

The uniform definability of each $T_n$, hence of $T = \bigcupup_n T_n$, follows from a straightforward translation of the above construction, similarly to in the proof of \cref{thm:cptt}.
Here in the inductive step, we apply the formulas defining the construction in \cref{thm:cptt} relativized to each $\Half^{\eval{P_n}_D}(D)$-block $A$ not containing a $P_n$-class,
as well as to the quotient median graph $D/\Half^{\eval{P_n}_D}(D)$
(with the linear ordering and $\#N$-coloring on the quotient constructed from that on $D$ via \cref{rmk:cptt-quot}).
\end{proof}

\section{Merging finitely many trees}\label{sec:merging_trees}

We now combine \cref{thm:tree-link-subalgebra,thm:cptt-union} to canonically produce a tree from a linked forest of finitely many trees,
equipped with auxiliary combinatorial input that is always available on a countable Borel equivalence relation (see again \cref{thm:Lomega1omega-borel-LO+color}).

\begin{theorem}
\label{thm:tree-link-treeable}
Let $X$ be a nonempty set,
$T \subseteq X^2$ be a locally finite forest with finitely many components,
and $\link = (\link_{A,B} : A \cong B)_{A, B \in X/T}$ be a link between the $T$-components.

Suppose we are also given a linear ordering, as well as a proper $\#N$-coloring of the intersection graph on the collection of all finite subsets, of the entire set $X$.

From this input structure, we may uniformly define (in $\@L_{\omega_1\omega}$) a tree on $X/\link$, or equivalently, a forest $F \subseteq X^2$ with the same connectedness relation as $T$ such that each $\link_{A,B}$ is an isomorphism between the $F$-components $A, B$.
\end{theorem}

\begin{proof}
Let $\LINK \subseteq \prod_{A \in X/T} A$ be the hypergraph of $\link$ (\cref{def:link}).
As described following \cref{thm:tree-link-subalgebra}, we may regard $\prod_{A \in X/T} A$ as the set of all transversals $\set{x_A}_{A \in X/T}$ of the connectedness relation $\#E_T$,
and $\LINK \subseteq \prod_{A \in X/T} A$ as consisting of the transversals which are orbits of $\link$, i.e., $\LINK \cong X/\link$.
From the given linear ordering as well as proper $\#N$-coloring of the intersection graph on all finite subsets of $X$, we may obtain the same auxiliary structure on $\prod_{A \in X/T} A$, hence on $\LINK$, via the usual lexicographical ordering (see \cref{thm:Lomega1omega-LO+color-finquot}).

Apply \cref{thm:tree-link-subalgebra} followed by \cref{thm:cptt-union} to obtain a tree $T'$ on the vertex set
$\ang{\LINK}_{T_\boxprod} \subseteq \prod_{A \in X/T} A$,
where
$T_\boxprod := \bigboxprod_{A \in X/T} \eval{T}_A$
is the product median graph of the $T$-components
(as in the proof of \cref{thm:tree-link-subalgebra}).
We have a canonical injection
$\iota : X/\link \cong \LINK `-> \ang{\LINK}_{T_\boxprod}$.
Using a Voronoi cell argument that is essentially Schreier's proof of the Nielsen--Schreier theorem,
and whose Borel version is given in \cite[3.3]{JKL:cber}
(see \cref{thm:nielsen-schreier} for the precise definable formulation),
we may define from $T'$ a tree on the image of $\iota$,
which pulls back along $\iota$ to a tree on $X/\link$.
To get the equivalent form of a forest $F$ with identical components on the $T$-components:
lift this tree on $X/\link$ to $X$,
and intersect with the connectedness relation of $T$.
\end{proof}

\section{Finite index over treeable}
\label{sec:treeable-finite-index}

From \cref{thm:tree-link-treeable}, our main result follows easily, via a reduction observed by Tserunyan \cite{tserunyan:thesis}:

\begin{theorem}
\label{thm:treeable-finite-index}
Let $E \subseteq X^2$ be a countable Borel equivalence relation on a standard Borel space $X$,
and let $F \subseteq E$ be a Borel subequivalence relation.
Suppose $F$ has finite index in $E$ and is Borel treeable.
Then $E$ is Borel treeable.
\end{theorem}

\begin{proof}
Recall \cref{def:link} of a link $\link$ between a family of sets, and its orbit equivalence relation.
By a \defn{Borel link $\link$ of the subequivalence relation $F \subseteq E$}, we mean a Borel family of links between the $F$-classes contained within each $E$-class.
The union of the orbit equivalence relations of these $E$-classwise links is then a partial subequivalence relation $\Link \subseteq E$;
we call this union $\Link$ the \defn{orbit equivalence relation} of the link $\link$ of the subequivalence relation $F \subseteq E$,
and we may take the Borelness of $\Link \subseteq E$ to be the formal meaning of ``Borel family of links''.
A \defn{Borel partial link of $F \subseteq E$} is defined analogously.
\footnote{In \cite{tserunyan:thesis}, the term ``link'' was used for what we are calling the orbit equivalence relation of a Borel link.}

Note that an arbitrary Borel finite partial subequivalence relation $\Link \subseteq E$ is the orbit equivalence relation of a partial link of $F \subseteq E$
iff each $\Link$-class $\linkclass \subseteq X$ meets each $F$-class contained within its $E$-class $[\linkclass]_E$ exactly once.
By \cite[7.3]{kechris-miller:topics}, there exists such an $\Link \subseteq E$ whose domain $Y \subseteq X$ meets every $E$-class.
By \cite[3.3]{JKL:cber}, treeability of $F$ implies treeability of $\eval{F}_Y$, and it suffices to show treeability of $\eval{E}_Y$ which then implies treeability of $E$.

Fix a Borel locally finite treeing $T \subseteq \eval{F}_Y$ by \cite[3.12]{JKL:cber}
(or, more simply, by replacing $(Y, \eval{E}_Y, \eval{F}_Y)$ by its amplification
$(Y \times \#N, \eval{E}_Y \times \#I_\#N, \eval{F}_Y \times \#I_\#N)$
(see \cite[3.23]{kechris:cber}),
so that $\eval{F}_Y$ becomes compressible,
hence degree $\le 3$ treeable by \cite[3.10]{JKL:cber}).
By applying \cref{thm:tree-link-treeable} $\eval{E}_Y$-classwise to $T$,
the Borel link $\link$ interdefinable with $\Link$ (see \cref{ex:Lomega1omega-link}),
a Borel linear ordering of $Y$,
and a proper $\#N$-coloring of the intersection graph on finite subsets of each $\eval{E}_Y$-class given by \cite[7.3]{kechris-miller:topics},
we obtain a treeing of the quotient equivalence relation $\eval{E}_Y/\Link$ on the standard Borel quotient space $Y/\Link$,
which is Borel by uniform definability of the construction (see \cref{thm:Lomega1omega-borel}).
The quotient map $Y ->> Y/\Link$ is a Borel reduction $\eval{E}_Y \le_B \eval{E}_Y/\Link$;
hence $\eval{E}_Y$ is Borel treeable again by \cite[3.3]{JKL:cber}.
\end{proof}

\begin{remark}
\label{rmk:treeable-finite-index-borel}
For readers who might prefer to work in the locally countable Borel context,
without using the formalism of $\@L_{\omega_1\omega}$ definability on countable structures,
we briefly indicate here how to translate the proof of \cref{thm:treeable-finite-index} into the locally countable Borel formalism. 

As shown in the proof of \cref{thm:treeable-finite-index},
it suffices to handle the case where $F$ is a finite index Borel subequivalence relation of $E$ on a standard Borel space $Y$,
$\link$ is a Borel link of $F \subseteq E$ with orbit equivalence relation $\Link \subseteq E$,
and there is a Borel locally finite treeing $T \subseteq F$ of $F$.

Let $\~Y$ be the standard Borel space of finite subsets $\linkclass \subseteq Y$ that are a transversal of the restriction of $F$ to a single $E$-class,
and let $\~E \subseteq \~Y^2$ be the natural equivalence relation induced by $E$,
i.e.,
\begin{eqaligned*}
\linkclass \mathrel{\~E} \linkclass'  \coloniff  [\linkclass]_E = [\linkclass']_E
\end{eqaligned*}
for $\linkclass, \linkclass' \in \~Y$.
For each $E$-class $C \in Y/E$,
as described following \cref{thm:tree-link-subalgebra},
we may regard the corresponding $\~E$-class
$\~C := \set{\linkclass \in \~Y | [\linkclass]_E = C}$
as an alternate representation of the product set
$\prod_{A \in C/F} A$;
and the subspace $Y/\Link \subseteq \~Y$ of $\Link$-classes,
i.e., orbits of $\link$,
corresponds to the hypergraph $\LINK$ of the link $\link$ (\cref{def:link}).
We thus denote $Y/\Link \subseteq \~Y$ also by $\LINK \subseteq \~Y$.
We let $T_\boxprod$ denote the $\~E$-classwise median graph given on each class $\~C \in \~Y/\~E$ as the product of the $T$-components in the corresponding $C \in Y/E$ (as in \labelcref{thm:tree-link-subalgebra}).
We also denote by $\ang{\LINK}_{T_\boxprod} \subseteq \~Y$ the $\~E$-classwise median subalgebra of $(\~Y, T_\boxprod)$ generated by $\LINK$.
Running the proof of \cref{thm:tree-link-subalgebra} $E$-classwise,
we obtain a sequence of Borel partial subequivalence relations
\begin{eqaligned*}
({=_\LINK}) \subseteq P_1 \subseteq P_2 \subseteq P_3 \subseteq \dotsb \subseteq \~E
\end{eqaligned*}
obeying the conclusions of \labelcref{thm:tree-link-subalgebra} on each $\~E$-class,
such that $\bigcupup_n P_n$ is the restriction of $\~E$ to $\ang{\LINK}_{T_\boxprod}$.

Now apply the proof of \cref{thm:cptt-union} $\smash{\eval{\~E}_{\ang{\LINK}_{T_\boxprod}}}$-classwise to the $P_n$'s,
where the auxiliary linear ordering and proper $\#N$-coloring of finite subsets are given by a Borel linear ordering on $\~Y$ and \cite[7.3]{kechris-miller:topics}.
The proof of \cref{thm:cptt-union} involves two different applications of \cref{thm:cptt}.
The first translates to directly applying \cref{thm:cptt}, i.e., \cite[3.5]{CPTT:quasitree}, classwise to the induced median graphing on a Borel subrelation of $\~E$ (whose classes are the $\Half^{\eval{P_n}_D}(D)$-blocks not containing a $P_n$-class, in the notation of \labelcref{thm:cptt-union}).
The second application of \cref{thm:cptt} in the proof of \labelcref{thm:cptt-union} is to a quotient median graph (of $\Half^{\eval{P_n}_D}(D)$-blocks);
rather than arguing that the auxiliary linear ordering and proper $\#N$-coloring on edges passes to this quotient graph, as in \cref{rmk:cptt-quot} used in \labelcref{thm:cptt-union}, one may also view this as a ``relativized'' version of the proof of \labelcref{thm:cptt} that only iterates through half-spaces in the subpocset ($\Half^{\eval{P_n}_D}(D)$) we are quotienting by.

The output of \cref{thm:cptt-union} is a Borel treeing of $\~E$ restricted to the complete section $\smash{\ang{\LINK}_{T_\boxprod}} \subseteq \~Y$.
It is easily seen that $(Y,E)$ is Borel bireducible to $(\~Y,\~E)$;
thus $E$ is Borel treeable by \cite[3.3]{JKL:cber}.
\end{remark}

\begin{remark}\label{rmk:invariant-free-complement}
Let $Y$, $E$, $F$, $T$ and $L$ be as in \cref{rmk:treeable-finite-index-borel}, and let $Q_n\subseteq F$ and $\~Q_n \subseteq E$ be defined as in \cref{rmk:Q_nDefinition}.
The proof of \cref{thm:treeable-finite-index} shows that each $Q_n$ is \emph{$L$-invariantly} freely complemented within $Q_{n+1}$. 
This follows from the proof of \cref{thm:tree-link-treeable} using the fact that, in the treeing $T'$ on $\langle \Lambda \rangle_{T_\boxprod}$, each $P_n$-class (and each $P_{n+1}$-class) is convex by \cref{thm:cptt-union}.
Indeed, when defining the treeing on $Y/L$ via the Voronoi cell argument as in the proof \cref{thm:tree-link-treeable}, by working stage-by-stage one may ensure that the classes of $\~{Q}_n/L$ and of $\~{Q}_{n+1}/L$ are both convex, hence the final lift to an $L$-invariant treeing of $F$ on $Y$ witnesses that $Q_n$ is $L$-invariantly freely complemented in $Q_{n+1}$.
\end{remark}

\section{Some consequences}
\label{sec:corollaries}

\subsection{Functoriality}\label{subsec:functoriality}

Our positive answer to \cref{question:intro} takes the form of the following functor.
To each triple $(X,I,(T_i)_{i\in I})$, where $X$ and $I$ are nonempty sets with $I$ finite, and each $T_i$ is a locally finite tree on $X$,
\cref{thm:tree-link-subalgebra} associates a triple $\mathsf{M}(X, I, (T_i)_{i\in I})=(M,X,P)$ with the following properties:
\begin{itemize}
\item $M$ is an interval-finite median algebra containing $X$ as a generating subset.
\item The identity map on $X$ extends uniquely to a median homomorphism from $M$ onto each input tree $(X,T_i)$, and these homomorphisms jointly separate points of $M$.
\item $P=(P_n)_{n\geq 1}$ is an increasing sequence of partial equivalence relations with union $M^2$.
Each $P_n$-class $C$ is a median subalgebra of $M$ with generating set $C\cap X$, and $C$ is hyperplane-finite as a median graph;
$\eval{P_{n-1}}_C$ is invariantly separable by half-spaces in $C$ for $n\geq 2$.
\end{itemize}

Concretely, $M$ is the median subalgebra generated by the diagonal in the product of the trees $T_i$, with $X$ identified with the diagonal via $x\mapsto(x)_{i\in I}$,
and with the coordinate projections giving the median homomorphisms onto the $T_i$'s.
The filtration $P$ is obtained by applying \cref{thm:tree-link-subalgebra} to the disjoint copies of these trees on $X\times I$, with the link $(x,i)\mapsto(x,j)$.

Such triples $(X,I,(T_i)_{i\in I})$ form the objects of the domain category for the functor $\mathsf{M}$.
The morphisms $(X,I,(T_i)_{i\in I}) -> (Y,J,(S_j)_{j\in J})$ are pairs $(f,\sigma)$ of bijections $f:X\to Y$ and $\sigma:I\to J$
such that $f$ is an isomorphism from $T_i$ to $S_{\sigma(i)}$ for every $i\in I$.
(Thus, the category is a groupoid.)

The objects of the codomain category are triples $(M,X,P)$, where $M$ is an interval-finite median algebra, $X\subseteq M$ is a generating subset, and $P=(P_n)_{n\geq 1}$ is an increasing sequence of partial equivalence relations with union $M^2$.
A morphism from $(M,X,P)$ to $(N,Y,Q)$ is a median isomorphism $M\to N$, taking $X$ to $Y$ and taking $P_n$ to $Q_n$ for every $n\geq 1$.

Writing $\mathsf{M}(Y,J,(S_j)_{j\in J})=(N,Y,Q)$, we assign to each morphism $(f,\sigma):(X,(T_i)_{i\in I})\to(Y,(S_j)_{j\in J})$ the median isomorphism $\mathsf{M}(f,\sigma):M\to N$
given by $(x_i)_{i\in I}\mapsto(f(x_{\sigma^{-1}(j)}))_{j\in J}$, which takes $P_n$ to $Q_n$ for every $n\geq 1$.
The following is now an immediate consequence of uniform definability of the construction in \cref{thm:tree-link-subalgebra}.

\begin{corollary}
\label{cor:tree-family-filtration}
The assignment $\mathsf{M}$ defines a functor between these categories (groupoids).
\end{corollary}

\subsection{Treeability of groups}
\cref{thm:treeable-finite-index} implies that several notions of treeability for groups are preserved under finite index extensions, even in the locally compact setting.

A locally compact second countable group $G$ is called \defn{Borel strongly treeable} if for every free Borel action of $G$ on a standard Borel space $X$, the associated orbit equivalence relation is Borel treeable;
$G$ is called \defn{measure strongly treeable} if, for every free Borel action of $G$ on $X$ and every Borel probability measure $\mu$ on $X$,
the orbit equivalence relation is Borel treeable after restricting to some $G$-invariant $\mu$-conull Borel subset of $X$;
$G$ is called \defn{strongly treeable} if the same condition holds, but with $\mu$ only ranging over $G$-invariant Borel probability measures; see \cite[Definition~B.1]{CGMTD}.

\begin{corollary}
\label{cor:strongly_treeable}
Let $G$ be a locally compact second countable group. 
If $G$ contains a finite index open subgroup $H \le G$ that is strongly treeable, then $G$ itself is strongly treeable.
The same holds for measure strong treeability and Borel strong treeability.
In all cases this implies that $G$ has fixed price, provided $G$ is unimodular.
\end{corollary}

This is new even in the case where $G$ is a countable discrete group.
The converse to \cref{cor:strongly_treeable}, i.e., that these three notions all descend to finite index open subgroups, is well-known; see \cite[Theorem~B.3]{CGMTD}.

\begin{proof}
Let $X$ be a standard Borel space with a free Borel action $G \actson X$, let $\#E_{G \actson X} \subseteq X^2$ be the orbit equivalence relation of this action, and let $Y \subseteq X$ be an $H$-invariant Borel subset such that the action $H \actson Y$ induces a Borel treeable orbit equivalence relation $\#E_{H \actson Y} \subseteq Y^2$.
By \cite{kechris:loccpt} and \cite[A.6]{CGMTD}, there is a Borel complete section $Z \subseteq Y$ of $\#E_{H \actson Y}$ meeting every $H$-orbit in countably many points, such that $\eval{\#E_{H \actson Y}}_Z$ is Borel treeable.
Then $\eval{\#E_{H \actson Y}}_Z$ is a finite index subequivalence relation of the countable Borel equivalence relation $\eval{\#E_{G \actson X}}_Z$, which is hence also Borel treeable by \cref{thm:treeable-finite-index}.
Hence so is $\#E_{G \actson GY}$, by adding an edge from each $x \in GY \setminus Z$ to any point in $Z \cap Gx$ chosen using the Lusin--Novikov theorem (see \cite[18.10]{bible}).
(Note that $GY \subseteq X$ is Borel, by translating $Y$ by finitely many coset representatives for $H \le G$.)

If now $H$ is Borel strongly treeable, then for any such free Borel $G$-space $X$, taking $Y = X$ above, we obtain that the orbit equivalence relation of $H \actson X$, hence also $G \actson X$, is Borel treeable; thus $G$ is Borel strongly treeable.
Similarly for (measure) strong treeability, given a measure $\mu$ on $X$, we have a $H$-invariant $\mu$-conull Borel $Y \subseteq X$ such that $\#E_{H \actson Y}$ is Borel treeable, whence the above gives that $GY \subseteq X$ is a $G$-invariant $\mu$-conull Borel subset such that $\#E_{G \actson GY}$ is Borel treeable.
\end{proof}

We also note that combining \cref{intr:thm:treeable-finite-index} with \cite[Proposition 25.6]{kechris-miller:topics} gives the cost of finite index extensions of treeable relations in the probability measure preserving setting.
This answers a special case of \cite[Problem 25.5]{kechris-miller:topics}:
for pmp countable Borel equivalence relations $F\subseteq E$ on $(X,\mu )$, if $F$ is treeable and of index $n$ in $E$ then $C_\mu (E)-1=\frac{1}{n}(C_\mu (F)-1)$.

\subsection{Treeability and homomorphisms} In the context of \cref{thm:treeable-finite-index}, the identity map on $X$ is a Borel homomorphism from $F$ to $E$ that descends to a finite-to-one, surjective map $X/F ->> X/E$ on the quotients.
The following consequence of \cref{thm:treeable-finite-index} shows that treeability descends through arbitrary Borel homomorphisms with these properties. 

\begin{corollary}
\label{cor:finite-to-one_extension}
Let $F$ and $F_0$ be countable Borel equivalence relations on standard Borel spaces $X$ and $X_0$, respectively.
Suppose that $\pi:X\to X_0$ is a Borel homomorphism $F -> F_0$ that descends to a finite-to-one, surjective map $X/F ->> X_0/F_0$ between the quotients.
If $F$ is Borel treeable, then $F_0$ is Borel treeable.
\end{corollary}
\begin{proof}
$\pi^{-1}(F_0) \subseteq X^2$ is a countable Borel equivalence relation containing the finite index Borel treeable subequivalence relation $F$, whence it is Borel treeable by \cref{thm:treeable-finite-index};
and $\pi$ is a Borel bireduction from $\pi^{-1}(F_0)$ to $F_0$, whence $F_0$ is Borel treeable by \cite[3.3]{JKL:cber}.
\end{proof}

\appendix
\section{Review of \texorpdfstring{$\mathcal{L}_{\omega_1\omega}$}{L\textomega\textoneinferior\textomega} definability and Borel structurability}
\label[appendix]{sec:Lomega1omega}

We review here the key concepts in the infinitary logic $\@L_{\omega_1\omega}$ and Borel structurability theory that underpin the formalism of ``uniform definability on countable structures'' used in this paper (see \cref{sec:prelim-cber}).
For more detailed references, see
\cite{marker:inflog},
%\cite{montalban:cst2},
\cite{chen:etalejt},
\cite{chen-kechris:structurable},
\cite{banerjee-chen:strinterp}.

\begin{definition}
\label{def:Lomega1omega}
The \defn{countably infinitary first-order logic} $\@L_{\omega_1\omega}$,
over a countable first-order language $\@L$,
is the extension of standard first-order logic with countably infinite conjunctions $\bigwedge_{i \in \#N}$ and disjunctions $\bigvee_{i \in \#N}$.
We require that every $\@L_{\omega_1\omega}$ formula $\phi(x_1, \dotsc, x_n)$ contain only finitely many free variables $x_1, \dotsc, x_n$.
A \defn{sentence} is a formula with no free variables;
a \defn{theory} is a set of sentences (called the axioms of the theory).
We will assume that all $\@L_{\omega_1\omega}$ theories are countable.
\footnote{This is needed for a clean correspondence with the Borel context.
Thus, one may as well replace a theory with a single sentence, namely the conjunction of its axioms; however, theories are more convenient to work with in practice.}

If $\@M$ is a model of a theory, on the underlying set $M$, and $\phi = \phi(x_1, \dotsc, x_n)$ is a formula, then we write
$\phi^\@M \subseteq M^n$
for the \defn{interpretation of $\phi$ in $\@M$},
i.e., the $n$-ary relation defined by $\phi$.
\end{definition}

\begin{example}
\label{ex:Lomega1omega-conngraph}
The following infinite formula in the language of graphs $\@L_\!{graph} = \set{G}$ asserts that $x, y$ belong to the same connected component:
\begin{eqaligned*}
\phi_\!{path}(x, y) := \bigvee_{n \in \#N} \exists x_0, \dotsc, x_n\, \paren[\big]{
    (x = x_0) \wedge (x_0 \mathrel{G} x_1) \wedge \dotsb \wedge (x_{n-1} \mathrel{G} x_n) \wedge (x_n = y)
}.
\end{eqaligned*}
Thus $\phi_\!{path}$ defines a binary relation, namely the connectedness relation $\#E_G$, in every graph $(X,G)$.
\end{example}

\begin{example}
\label{ex:Lomega1omega-graph-metric}
The relations ``$d_G(x, y) \le n$'' in a graph may be defined by formulas $\phi_\!{\$n-path}(x, y)$ analogous to $\phi_\!{path}$, with the above disjunction truncated at $n$.

Using such formulas, one may in turn define the betweenness relation
``$z \in [x, y]$''
(\cref{def:graph-interval}),
and hence the graph of the median operation
``$\ang{x, y, z} = m$''
in a median graph
(\cref{def:medgph}).
\end{example}

\begin{definition}
\label{def:Lomega1omega-def}
Given an $\@L_{\omega_1\omega}$ theory $\@T$,
and an assignment to each model $\@M$ of an $n$-ary relation
$R_\@M \subseteq M^n$,
we say these relations are \defn{uniformly definable (in $\@L_{\omega_1\omega}$)} if there is a \emph{single} formula $\phi(x_1, \dotsc, x_n)$ (in the language of $\@T$) defining $R_\@M$ in all models $\@M$.

Similarly, an assignment of $n$-ary functions $f_\@M : M^n -> M$ is \defn{uniformly definable} if the graphs of the $f_\@M$ are uniformly definable as $(n+1)$-ary relations.

If in each model $\@M$ of $\@T$, we may uniformly define a family of relations and functions, and these fit together to form a model $\@M'$ of another theory $\@T'$, then we say that we may \defn{uniformly define a model of $\@T'$ in every model of $\@T$}.

Thus for instance, the preceding example shows that we may uniformly define the induced median algebra structure on every median graph.
Conversely, it is easily seen that we may uniformly define the edge relation in interval-finite median algebras (see \cref{thm:medalg-medgph}).
Moreover, these two definitions are inverses of each other: the edge relation in the median algebra defined from a median graph recovers the original graph, and vice-versa.
We express this by saying that (the theories of) median graphs and interval-finite median algebras are \defn{interdefinable}.
\end{definition}

\begin{example}
\label{ex:Lomega1omega-link}
Let $X$ be a set equipped with an equivalence relation $F$ together with a link
$\link = (\link_{A,B} : A \cong B)_{A, B \in X/F}$
between the $F$-classes (\cref{def:link}).
This may be treated as a structure in the language
$\@L_\!{linkedequiv} = \set{F, \link}$,
where $\link$ is treated as a binary function
$(x,y) |-> \link_{[x]_F, [y]_F}(x)$.
Alternatively, we may encode $\link$ by its orbit equivalence relation
$\Link \subseteq X^2 = (\bigsqcup_{A \in X/F} A)^2$;
thus $\Link$ is a binary relation.
These two structures representing linked equivalence relations,
$(X, F, \link)$ and $(X, F, \Link)$,
are interdefinable, via formulas that may be read off of \cref{def:link}.
\end{example}

The following formalizes the notion of a ``locally countable Borel structure'', and its connection to definability in countable structures:

\begin{definition}
\label{def:Lomega1omega-borel}
Let $(X,E)$ be a countable Borel equivalence relation, $\@T$ be an $\@L_{\omega_1\omega}$ theory.
A \defn{Borel structuring $\@M$ of $E$ by models of $\@T$} is a ``Borel assignment'' to each $E$-class $C \in X/E$ of a countable model $\@M_\@C$ of $\@T$ on the underlying set $C$.

Here ``Borel assignment'' may be taken to mean that the disjoint union of the $\@M_C$'s is a Borel structure on $X$, or equivalently (if $E$ is aperiodic) that any Borel family of enumerations
$(f_x : \#N \cong [x]_E)_{x \in X}$
from the Lusin--Novikov theorem yields a Borel map
$x |-> f_x^{-1}(\@M_{[x]_E})$
from $X$ to the standard Borel space of $\@L$-structures on $\#N$.
(See \cite[3.1]{banerjee-chen:strinterp}.)

Thus for instance, a Borel treeing of $E$ is a structuring by (models of the theory of) trees.
\end{definition}

\begin{proposition}[see \cite[3.4]{banerjee-chen:strinterp}]
\label{thm:Lomega1omega-borel}
Let $(X,E)$ be a countable Borel equivalence relation,
$\@M = (\@M_C)_{C \in X/E}$ be a Borel structuring of $E$ by models of some $\@L_{\omega_1\omega}$ theory $\@T$.
Suppose that we may uniformly define, in each (countable) model $\@N$ of $\@T$, a model $\@N'$ of some other theory $\@T'$.
Then the family of models $(\@M_C')_C$ defined from the $\@M_C$'s yields a Borel structuring of $E$ by models of $\@T'$.
\end{proposition}

This gives a sufficient condition for a combinatorial construction on countable structures to yield a Borel construction.
It is not necessary however: certain Borel constructions are available ``for free'' on countable Borel equivalence relations, e.g., a Borel linear order, despite not being uniformly definable in $\@L_{\omega_1\omega}$.
In \cite[\S4]{banerjee-chen:strinterp}, an analysis is given of several possible ``bases'' for such non-canonical Borel constructions.
In this paper, we choose to work with the following ``basis'', which is the uniform definability formulation of a well-known lemma of Kechris--Miller \cite[7.3]{kechris-miller:topics}:

\begin{proposition}[see \cite[4.7, 3.27, 3.29]{banerjee-chen:strinterp}]
\label{thm:Lomega1omega-borel-LO+color}
A construction of a model of a theory $\@T'$ from every model of a theory $\@T$ may be carried in the locally countable Borel context (i.e., on every Borel structuring of a countable Borel equivalence relation by models of $\@T$)
iff it may be uniformly defined from a model of $\@T$ \emph{expanded with two pieces of auxiliary structure}: a linear ordering, and a proper $\#N$-coloring of the intersection graph on all finite subsets.
\end{proposition}

(Formally, the $\#N$-coloring of finite subsets here should be treated as a sequence of definable subsets of an imaginary sort, namely the sort of finite subsets; see \cref{ex:Lomega1omega-finpow} below.)

This result licenses the use of well-known results in Borel combinatorics in the context of uniform definability on countable structures.
For instance, we have the definable version of \cite[3.3(iii)]{JKL:cber},
which is also essentially Schreier's proof of the Nielsen--Schreier theorem
(see \cite[3.8]{lyndon-schupp:combgrp}):

\begin{proposition}[see \cite[3.3]{JKL:cber}]
\label{thm:nielsen-schreier}
Let $X$ be a set, $T \subseteq X^2$ be a tree, and $E \subseteq X^2$ be an equivalence relation (with no connection to $T$).
Suppose we are also given a proper $\#N$-coloring of the intersection graph on finite subsets of $X$ (in fact, a coloring of the common-source graph on edges of $T$ is enough).
From this input structure, we may uniformly define a forest $F \subseteq X^2$ generating $E$.
\end{proposition}

The following is likewise the uniform definability formulation of \cite[4.5]{KST:chromatic}:

\begin{proposition}[Kechris--Solecki--Todorcevic; see \cite[4.27]{banerjee-chen:strinterp}]
\label{thm:KST-lfgraph-color}
Let $(X,G)$ be a locally finite graph.
Suppose we are also given a linear ordering, as well as a proper $\#N$-coloring of the intersection graph on finite subsets, of $X$.
From this input, we may uniformly define a proper $\#N$-coloring of $G$.
\end{proposition}

In more involved multi-step combinatorial constructions,
such as those appearing in \cref{thm:tree-link-subalgebra,thm:cptt-union},
it is necessary to allow operations that change the underlying set.
In order to formalize ``uniformly definable'' such constructions,
we need the following generalization of formulas:

\begin{definition}
\label{def:Lomega1omega-imag}
Let $\@T$ be an $\@L_{\omega_1\omega}$ theory.
An \defn{imaginary sort (in $\@L_{\omega_1\omega}$) over $\@T$} is a \emph{formal} definable quotient of a countable disjoint union of definable relations in models of $\@T$.
That is, it is a syntactic expression of the form
\begin{eqaligned*}
\Phi = \paren[\Big]{\bigsqcup_{i \in I} \phi_i} \Big/ \paren[\Big]{\bigsqcup_{i, j \in I} \epsilon_{i,j}},
\end{eqaligned*}
where $I$ is a countable index set,
each $\phi_i$ is a formula in some number $n_i$ of variables,
each $\epsilon_{i,j}$ is a formula in $n_i+n_j$ variables,
and the $\epsilon_{i,j}$ define the pieces of an equivalence relation on
$\bigsqcup_{i \in I} \phi_i^\@M \subseteq \bigsqcup_{i \in I} M^{n_i}$
in every model $\@M$ of $\@T$;
the quotient set
\begin{eqaligned*}
\Phi^\@M := \paren[\Big]{\bigsqcup_{i \in I} \phi_i^\@M} \Big/ \paren[\Big]{\bigsqcup_{i, j \in I} \epsilon_{i,j}^\@M}
\end{eqaligned*}
is then called the \defn{interpretation of $\Phi$ in the model $\@M$}.

A \defn{subsort} $\Psi \subseteq \Phi$ is another imaginary sort whose interpretation in every model $\@M$ is a subset $\Psi^\@M \subseteq \Phi^\@M$.
Up to logical equivalence, such $\Psi$ is of the form
$\Psi = (\bigsqcup_i \psi_i)/(\bigsqcup_{i,j}, \epsilon_{i,j}')$
where the $\psi_i$'s define subsets of the $\phi_i$'s such that $\bigsqcup_i \psi_i \subseteq \bigsqcup_i \phi_i$ is invariant under the equivalence relation $\bigsqcup_{i,j} \epsilon_{i,j}$,
and the $\epsilon_{i,j}'$'s are the restrictions of the $\epsilon_{i,j}$'s to this subset.

Extending \cref{def:Lomega1omega-def}, we say that an assignment of relations $R_\@M \subseteq \Phi^\@M$ on an imaginary sort $\Phi$ in models $\@M$ of $\@T$ is \defn{uniformly definable} if it is defined by some subsort $\Psi \subseteq \Phi$.

See \cite[Ch.~5]{hodges:model} for background on imaginaries in finitary first-order logic,
and \cite[\S8--10]{chen:etalejt} for background on imaginaries in $\@L_{\omega_1\omega}$.
\end{definition}

\begin{example}
\label{ex:Lomega1omega-finpow}
For a set $X$, we may treat the finite powerset
$[X]^{<\omega} := \set{A \subseteq X | \abs{A} < \aleph_0}$
as the quotient of the set of injective finite sequences,
which is a definable subset of
$X^{<\omega} = \bigsqcup_{n \in \#N} X^n$,
by the action of the symmetric group $S_n$ for each $n$.
That is, $[X]^{<\omega}$ is the interpretation in $X$ of the imaginary
\begin{eqaligned*}
\Phi_\!{finpow} := \paren[\Big]{\bigsqcup_{n \in \#N} \phi_n} \Big/ \paren[\Big]{\bigsqcup_{m, n \in \#N} \epsilon_{m,n}}
\end{eqaligned*}
(in the empty language $\@L = \emptyset$) where
\begin{eqaligned*}
\phi_n(x_1, \dotsc, x_n) &:= \bigwedge_{1 \le i < j \le n} (x_i \ne x_j),  \\
\epsilon_{m,n}(x_1, \dotsc, x_m, y_1, \dotsc, y_n) &:= \begin{cases*}
    \bot & if $m \ne n$,  \\
    \bigvee_{\sigma \in S_n} \bigwedge_{1 \le i \le n} (x_i = y_{\sigma(i)}) & if $m = n$.
\end{cases*}
\end{eqaligned*}
Thus, a subsort $\Psi \subseteq \Phi_\!{finpow}$ is given by a family of $n$-ary formulas $\psi_n(x_1, \dotsc, x_n)$ for each $n \in \#N$ defining relations that are $S_n$-invariant, hence descend to the quotient.
\end{example}

For instance, the auxiliary structure of an $\#N$-coloring of finite subsets in \cref{thm:Lomega1omega-borel-LO+color} is formally given by countably many definable subsets
$C_0, C_1, \dotsc \subseteq \Phi_\!{finpow}$
(the color classes).

A basic fact used throughout Borel combinatorics is the ability to pick a point in a finite set in a Borel way;
this allows finite-to-one quotients in the Borel context to produce a standard Borel space.
The following is the uniform definability formulation of this technique:

\begin{proposition}
\label{thm:Lomega1omega-LO+color-finquot}
Let $X$ be a set equipped with a linear ordering $<$ and a proper $\#N$-coloring of the intersection graph on finite subsets (as in \cref{thm:Lomega1omega-borel-LO+color}),
and let $\phi_i \subseteq X^{n_i}$ such that each $n_i > 0$ and $\epsilon_{i,j} \subseteq X^{n_i+n_j}$ such that $E := \bigsqcup_{i,j} \epsilon_{i,j}$ is a \emph{finite} equivalence relation on $\bigsqcup_i \phi_i$.
Then on the imaginary sort $\Phi = (\bigsqcup_i \phi_i) / (\bigsqcup_{i,j} \epsilon_{i,j})$, we may also uniformly define a linear ordering and a proper $\#N$-coloring of the intersection graph on finite subsets.
\end{proposition}
\begin{proof}
Using the provided $<$, we may lexicographically order each $X^{n_i}$;
then we may order $\bigsqcup_i \phi_i$ by taking the ordered sum of the $\phi_i \subseteq X^{n_i}$.
Since $E$ has finite classes, we may order $\Phi = (\bigsqcup_i \phi_i)/E$ according to the least elements in each class.
Given a finite subset $F \subseteq \Phi$, we may color it $\ang{m, k}$ where
$\ang{\blank,\blank} : \#N^2 -> \#N$ is some fixed bijection,
$m$ is the color assigned to the finite subset $\norm{F} \subseteq X$ consisting of all coordinates of all tuples in all $E$-classes in $F$,
and $k$ is the index of $F$ in the lexicographical ordering on the finite set of all finite subsets $F' \subseteq \Phi$ with $\norm{F'} = \norm{F}$.
\end{proof}

% remove MR number from references
\def\MR#1{}
\bibliographystyle{amsalpha}
\bibliography{refs}

\parindent=0pt\relax
\end{document}